\documentclass{article}

\usepackage{multicol}
\usepackage{multirow}
\usepackage{fullpage}
\usepackage{amsthm}
\usepackage{latexsym}
\usepackage{amsmath,amscd}
\usepackage{amsfonts}
\usepackage{amssymb}
\usepackage[T1]{fontenc}
\usepackage{verbatim} 
\usepackage[all]{xy}
\usepackage[french,english]{babel}
\usepackage{blkarray}
\usepackage{mathrsfs}

\newcommand{\ls}{\leqslant}
\newcommand{\gs}{\geqslant}

\newcommand{\bX}{\mathbb{X}}
\newcommand{\bT}{\mathbb{T}}
\newcommand{\bN}{\mathbb{N}}
\newcommand{\bP}{\mathbb{P}}
\newcommand{\bC}{\mathbb{C}}
\newcommand{\cO}{\mathcal{O}}
\newcommand{\cF}{\mathcal{F}}
\newcommand{\cE}{\mathcal{E}}
\newcommand{\ep}{\varepsilon}
\newcommand{\cG}{\mathcal{G}}

\newcommand{\bI}{\mathbb{I}}
\newcommand{\Id}{Id}
\newcommand{\Om}{\Omega}
\newcommand{\tOm}{\widetilde{\Omega}}
\newcommand{\Si}{\Sigma}
\newcommand{\tSi}{\widetilde{\Sigma}}
\newcommand{\la}{\lambda}
\newcommand{\tva}{\tilde{\varphi}}
\newcommand{\tla}{\widetilde{\lambda}}

\newcommand{\cS}{\mathcal{S}}

\newcommand{\Ga}{\Gamma}
\newcommand{\al}{\alpha}
\newcommand{\om}{\omega}
\newcommand{\fU}{\mathfrak{U}}
\newcommand{\tphi}{\tilde{\varphi}}
\newcommand{\vphi}{\varphi}
\newcommand{\La}{\Lambda}
\newcommand{\si}{\sigma}
\newcommand{\tom}{\widetilde{\omega}}
\newcommand{\be}{\beta}
\newcommand{\bA}{\mathbb{A}}
\newcommand{\ccX}{\mathscr{X}}
\newcommand{\De}{\Delta}
\newcommand{\bL}{\mathbb{L}}
\newcommand{\de}{\delta}
\newcommand{\Eul}{\mathscr{E}}

\DeclareMathOperator{\Gl}{Gl}
\DeclareMathOperator{\Mat}{Mat}
\DeclareMathOperator{\rk}{rk}

\DeclareMathOperator{\codim}{codim}
\DeclareMathOperator{\bB}{Bs}

\DeclareMathOperator{\n}{n}
\DeclareMathOperator{\s}{s}
\DeclareMathOperator{\w}{w}
\DeclareMathOperator{\im}{im}
\DeclareMathOperator{\Grass}{Grass}
\DeclareMathOperator{\pr}{pr}
\DeclareMathOperator{\Hom}{Hom}

\DeclareMathOperator{\nz}{nz}

\theoremstyle{plain}
\newtheorem{theorem}{Theorem}[section]
\newtheorem{lemma}[theorem]{Lemma}
\newtheorem{proposition}[theorem]{Proposition}
\newtheorem{corollary}[theorem]{Corollary}

\newtheorem{theointro}{Theorem}

\theoremstyle{definition}
\newtheorem{definition}[theorem]{Definition\rm}
\theoremstyle{remark}
\newtheorem{remark}[theorem]{\it Remark\/}

\newtheorem{claim}{\it Claim\/}

\begin{document}

\title{Explicit symmetric differential forms on complete intersection varieties and applications
}


\author{Damian Brotbek     
}
\maketitle

\begin{abstract}
In this paper we study the cohomology of tensor products of symmetric powers of the cotangent bundle of complete intersection varieties in projective space. We provide an explicit description of some of those cohomology groups in terms of the equations defining the complete intersection. We give several applications. First we prove a non-vanishing result, then we give a new example illustrating the fact that the dimension of the space of holomorphic symmetric differential forms is not deformation invariant. Our main application is the construction of varieties with ample cotangent bundle, providing new  results towards a conjecture of Debarre. 
\end{abstract}
\tableofcontents

\section{Introduction}
Varieties with ample cotangent bundle have many interesting properties, however, relatively few examples of such varieties are know (see \cite{Sch85}, \cite{Som84}, \cite{Deb05}, \cite{Bro14}). Debarre conjectured in \cite{Deb05} that: \emph{ if $X$ is a general complete intersection variety in $\bP^N$ of multidegree high enough and such that $\dim X\ls \codim_{\bP^N}X$ then the cotangent bundle of $X$ should be ample.} The study of this conjecture was the starting point of the present work. In \cite{Bro14} we were able to prove this conjecture when $\dim X=2,$ using Voisin's variational method and inspired by the work of Siu \cite{Siu04} and the work of Diverio Merker and Rousseau  \cite{DMR10}. However we were not able to make this strategy work completely in higher dimensions. \\

The construction of varieties with positive cotangent bundle is closely related to the construction of symmetric differential forms on it. In fact, if one wants to prove that a given variety has ample cotangent bundle, it is natural to start by producing many symmetric differential forms, to  be more precise, this means proving that the cotangent bundle is big. In general, this is already a highly non-trivial question and this leads to very interesting considerations. In that direction we would like to mention  the recent work of Brunebarbe, Klingler and Totaro \cite{BKT13} as well as the work of Roulleau and Rousseau \cite{R-R14}.\\

However, ampleness is a much more restrictive condition than bigness, in some sense, bigness only requires a quantitative information on the number of symmetric differential forms whereas ampleness requires a more qualitative information on the geometry of the symmetric differential forms. The most natural way to produce symmetric differential forms is to use Riemann-Roch theorem or a variation of it. For instance  under the hypothesis of Debarre's conjecture, one can use Demailly's holomorphic Morse inequality, in the spirit of Diverio's work \cite{Div08} and \cite{Div09}, to prove that the cotangent bundle is big (see \cite{Bro14}). Nevertheless, this approach doesn't give much information on the constructed symmetric differential forms besides its existence. One can wonder if it is possible, given a complete intersection variety $X$ in $\bP^N$ to write down explicitly the equation of a symmetric differential form on $X$ (if such an object exists).\\

If $X$ is a curve in $\bP^2$ of genus greater than $1$, this is a very classical exercise. In higher dimensions, very few results towards that problem are known. Br\"uckmann \cite{Bru85} constructed an example of a symmetric differential form on a complete intersection in $\bP^4$ given by Fermat type equations, and more recently Merker \cite{Mer13} was able to study examples of symmetric differential forms on complete intersection variety in $\bP^4$ in the spirit of work of Siu and Yeung \cite{S-Y96} (see also \cite{Mer14} for related results for higher order jet differential equations). \\

The aim of this paper is to develop a cohomological framework which will enable us to describe the space of holomorphic symmetric differential forms on a complete intersection variety in $\bP^N$ in terms of its defining equations, and to give several applications. The outline of the paper is as follows.\\

In Section \ref{MainSection} we prove the main technical result of this work. In view of the possible generalizations to higher order jet differential as well as for its own sake, we will not only study the space of symmetric differential forms, but also the more general spaces $H^i(X,S^{\ell_1}\Om_X\otimes\cdots\otimes S^{\ell_k}\Om_X)$ for a complete intersection variety $X$ in $\bP^N.$ Recall the following vanishing result of Br\"uckman and Rackwitz.

\begin{theorem}[Br\"uckmann-Rackwitz \cite{B-R90}]\label{B-RVanishing}
Let $X\subseteq \bP^N$ be a complete intersection of dimension $n$ and codimension $c$. Take integers $j,\ell_1,\dots,\ell_k\geqslant 0$. If $j<n-\sum_{i=1}^k\min\{c,\ell_i\}$ then
$$H^j(X,S^{\ell_1}\Om_X\otimes\cdots \otimes S^{\ell_k}\Om_X)=0.$$
\end{theorem}
It is natural to look at what happens in the case $j= n-\sum_{i=1}^k\min\{c,\ell_i\}$ in the above theorem. Our result in that direction is the following.
\begin{theointro}\label{MainTheorem}
Let $c,N,e_1,\dots,e_c\in\bN^*$, set $n=N-c$. Let $X\in \bP^N$ be a smooth complete intersection variety of codimension $c$, dimension $n$, defined by the ideal $(F_1,\dots, F_c)$, where $F_i\in H^0(\bP^N,\cO_{\bP^N}(e_i))$. Take integers $\ell_1,\dots,\ell_k\geqslant c$ take an integer $a<\ell_1+\cdots +\ell_k-k$, let $q:=n-kc$ and $b:=(k+1)\sum_{i=1}^ce_i$.  Then one has a commutative diagram
$$
\xymatrix{
H^q\left(X,S^{\ell_1}\Om_X\otimes\cdots \otimes S^{\ell_k}\Om_X(a)\right)\ar[r] \ar[d]_{\Eul} \ar[dr]^{\varphi} & H^N\left(\bP^N,S^{\ell_1-c}\Om_{\bP^N}\otimes\cdots \otimes S^{\ell_k-c}\Om_{\bP^N}(a-b)\right)\ar[d]^{\Eul}\\
H^q\left(X,S^{\ell_1}\tOm_X\otimes\cdots \otimes S^{\ell_k}\tOm_X(a)\right)\ar[r]^{\tva\ \ \ \ \ } & H^N\left(\bP^N,S^{\ell_1-c}\tOm_{\bP^N}\otimes\cdots \otimes S^{\ell_k-c}\tOm_{\bP^N}(a-b)\right)\\
}
$$
Such that all the arrows are injective and such that:
\begin{enumerate}
\item $\im \tva =\bigcap_{i=1}^c \left(\ker(\cdot F_i)\bigcap_{j=1}^k\ker(\cdot dF_i^{\{j\}})\right)$. 
\item$ \im \varphi=\im \tva \cap \im\Eul.$
\end{enumerate}
\end{theointro}
\begin{remark} The bundle $\tOm$ is described in Section \ref{TidleOmega} and the  different maps arising in the statement are described in Section \ref{SectionStatements}
\end{remark}
The important thing to note in that result, is that it gives a way of describing the vector space $H^q\left(X,S^{\ell_1}\Om_X\otimes\cdots \otimes S^{\ell_k}\Om_X(a)\right)$ as a sub-vector space of $H^N\left(\bP^N,S^{\ell_1-c}\tOm_{\bP^N}\otimes\cdots \otimes S^{\ell_k-c}\tOm_{\bP^N}(a-b)\right)$, that this last space is easily described, and that one can precisely determine, in terms of the defining equations of $X$ what is the relevant sub-vector space. Therefore this result (and the more general statements in Theorem \ref{LimiteSetting} and Theorem \ref{LimitePair}) should be understood as our main theoretical tool to construct symmetric differential forms on complete intersection varieties. \\

In Section \ref{SectionApplications} we provide the first applications of Theorem \ref{MainTheorem}. First we describe how Theorem \ref{MainTheorem} can be used very explicitly   in \v{C}ech cohomology. Then we illustrate this by treating the case of curves in $\bP^2$. After that (Proposition \ref{example}) we prove that the result of Br\"uckmann and Rackwitz is optimal by providing the following non-vanishing result.  
\begin{theointro}\label{theoB}
Let $N\geqslant 2$, let $0\leqslant c<N$. Take integers $\ell_1,\dots,\ell_k\gs 1$ and $a<\ell_1+\cdots+\ell_k-k$. Suppose  $0\leqslant q:= N-c-\sum_{i=1}^k\min\{c,\ell_i\}$. Then, there exists a smooth complete intersection variety in $\bP^N$ of codimension $c$, such that 
$$H^q(X,S^{\ell_1}\Om_X\otimes\cdots\otimes S^{\ell_k}\Om_X(a))\neq 0.$$
\end{theointro} 
Then, in Corollary \ref{CoroFamily}, we provide a new example of a family illustrating the fact that the dimension of the space of holomorphic symmetric  differential forms is not deformation invariant.
\begin{theointro}\label{theoC}
For any $n\geqslant 2$, for any $m\geqslant 2$, there is a family of varieties $\mathscr{Y}\to B$ over a curve, of relative dimension $n$, and a point $0\in B$ such that for generic $t\in B,$ $$h^0(Y_0,S^m\Om_{Y_0})> h^0(Y_t,S^m\Om_{Y_t}).$$
\end{theointro}
This phenomenon has already been studied (see for instance \cite{Bog78}, \cite{BDO08} and \cite{R-R14}) and is well known. However, this example shows that invariance fails for any $m\geqslant 2$, whereas  the other known examples  (based on intersection computations) provide the result for $m$ large enough.\\

In Section \ref{SectionAmplitude} we provide our main application, which is a special case of Debarre's conjecture.
\begin{theointro}\label{theoD}
Let $N,e\in\bN^*$ such that $e\geqslant 2N+3$. If $X\subseteq \bP^N$ be a general complete intersection variety of  multidegree $(e,\dots, e)$, such that $\codim_{\bP^N}X\gs 3\dim X-2$, then $\Om_X$ is ample.
\end{theointro}
To our knowledge, this is the first higher dimensional result towards Debarre's conjecture. The proof of this statement does not rely on the variational method neither does it need the Riemann-Roch theorem nor Demailly's holomorphic Morse inequalities. The idea is to use the results of Section \ref{MainSection} to construct one very particular example of a smooth complete intersection variety in $\bP^N$ (with prescribed dimension and multidegree) whose cotangent bundle is ample. Then by the openness property of ampleness, we will deduce that the result holds generically. Such an example is produced by considering intersections of deformations of Fermat type hypersurfaces.\\

\emph{Notation and conventions:}  In this paper, we will be working over he field of complex numbers $\bC$. Given a smooth projective variety $X$ and a vector bundle $E$ on $X$, we will denote by $S^mE$ the $m$-th symmetric power of $E$,  we will denote by $\bP(E)$ the projectivization of rank one quotients of $E$, we will denote the tangent bundle of $X$ by $TX$ and the cotangent bundle of $X$ by $\Om_X.$ Moreover we will denote by $\pi_X:\bP(\Om_X)\to X$ the canonical projection. Given a line bundle $L$ on $X$ and an element $\si\in H^0(X,L)$ we will denote the zero locus of $\si$ by $(\si=0),$ and the base locus of $L$ by $\bB(L)=\bigcap_{\si\in H^0(X,L)}(\si=0).$ \\ Given any $m\in \bN$, we will denote by $\bC[Z_0,\dots,Z_N]_m$ the set of homogenous polynomials of degree $m$ in $N+1$ variables and by $\bC[z_1,\dots,z_N]_{\ls m}$ the set of polynomials of degree less or equal to $m$ in $N$ variables.  Given any set $E\subseteq \bN$ and any $k\in \bN$ we will write $E^k_{\neq}:=\{(i_1,\dots,i_k)\ \text{such that}\  i_j\neq i_\ell \ \text{if} \ j\neq \ell\}$ and $E^k_{<}:=\{(i_1,\dots,i_k)\ \text{such that}\  i_j<i_\ell \ \text{if} \ j< \ell\}.$\\ Also, we will say that a property holds for a \og \emph{general}\fg\   or a \og\emph{generic}\fg\ member of a family $\ccX\stackrel{\rho}{\to} T$ if there exists a Zariski open subset $U\subseteq T$ such that the property holds for $\rho^{-1}(t)$ for any $t\in U$. \\


\emph{Acknowledgments.} This work originated during the author's phd thesis under the supervision of Christophe Mourougane. We thank him very warmly for his guidance, his time and all the discussions we had. We also thank Junjiro Noguchi and Yusaku Tiba for listening through many technical details. We  thank Jo\"el Merker for his many encouragements and for all the interest he showed in this work. We also thank Lionel Darondeau for motivating discussions and for his  suggestions about the presentation of this paper. 

\section{Cohomology of  symmetric powers of the cotangent bundle}\label{MainSection}
\subsection{The tilde cotangent bundle}\label{TidleOmega}

It will be convenient for to use the $\tOm$ bundle, studied in particular by Bogomolov and DeOliveira in \cite{BDO08}, but also by Debarre in \cite{Deb05}. In some way, the bundle $\tOm$ will allow us to work naturally in homogenous coordinates. Let us recall some basic facts about this bundle. Consider $\bP^N=P(\bC^{N+1})$ with its homogenous coordinates $[Z_0,\dots, Z_N].$ Let $X\subseteq \bP^N$ be a smooth subvariety. We denote by $\gamma_X$ the Gauss map
\begin{eqnarray*}
\gamma_X:X&\to &\Grass(n,\bP^N)\\
x&\mapsto& \bT_xX
\end{eqnarray*} 
where $\bT_xX\subset \bP^N$ is the embedded tangent space of $X$ at $x$, and where $\Grass(n,\bP^N)$ denotes the grassmannian of $n$-dimensional linear projective subspace of $\bP^N.$ Let $\cS_{n+1}$ denote the tautological rank $n+1$ vector bundle on $\Grass(n,\bP^N).$ Then define $$\tOm_X:=\gamma_X^*\cS_{n+1}^\vee\otimes \cO_X(-1).$$  
We will refer to this bundle as the \emph{tilde cotangent bundle} of $X$, and a holomorphic section of $S^m\tOm_X$ will be called a \emph{tilde symmetric differential form.} Observe that one has a natural identification
$$\tOm_{\bP^N}\cong \bC^{N+1}\otimes \cO_{\bP^N}(-1)\cong \bigoplus_{i=0}^N\cO_{\bP^N}(-1)dZ_i.$$
 Therefore given any homogenous degree $e$ polynomial $F\in \bC[Z_0,\dots,Z_N]$ one can define a map 
 \begin{eqnarray}
 \cdot dF: \cO_{\bP^N}(-e)&\to & \tOm_{\bP^N}\\ \label{foisdF}
 g&\mapsto & g\cdot \sum_{i=0}^N\frac{\partial F}{\partial Z_i}dZ_i.\nonumber
 \end{eqnarray} 
 One easily verifies  that if $X\subseteq \bP^N$ is a smooth subvariety and if $F$ defines a hypersurface $H$ such that $Y:=X\cap H$ is a smooth hypersurface in $X$ then the above map fits into the following exact sequence,
 \begin{eqnarray}0\to \cO_{Y}(-e)\stackrel{\cdot dF}{\to}\tOm_X|_Y\to \tOm_Y\to 0.\label{TildeConormalExactSequence}\end{eqnarray}
 We will refer to it as the  \emph{tilde conormal exact sequence}. On the other hand, let $\widehat{X}\subseteq \bC^{N+1}\setminus \{0\}$ be the cone above $X$, and let $\rho_X=\widehat{X}\to X$ be the natural projection. Observe that $\rho_X^*\gamma_X^*\cS_{n+1}=T\widehat{X}.$ The differential $d\rho_X:T\widehat{X}\to \rho_X^*TX$ is not invariant under the natural $\bC^*$ action on $\widehat{X}$ because 
 for any $x\in \widehat{X}$, any $\xi\in T_x\widehat{X}$ and any $\lambda\in \bC^*$, $d\rho_{X,\lambda x}\xi=\frac{1}{\lambda}d\rho_{X,x}\xi$.  We can easily compensate this by a simple twist by $\cO_X(-1)$ as in the following 
 \begin{eqnarray*}
 \gamma_X^*\cS_{n+1,x}&\to & T_xX\otimes \cO_{X,x}(-1)\\
 (x,\xi)&\mapsto& (x,d\rho_{X,x}\xi\otimes x).
 \end{eqnarray*}
 This yields an exact sequence $0\to \cO_X(-1)\to \gamma_X^*\cS_{n+1}\to TX(-1)\to 0$ which we twist and dualize to get 
 \begin{eqnarray}
 0\to \Om_X\stackrel{\Eul}{\to} \tOm_X \to \cO_X\to 0.\label{EulerExactSequence}
 \end{eqnarray}
Will refer to it as the \emph{Euler exact sequence}. Note that the map $\Eul$ can be understood very explicitly. Indeed, if we consider the chart $\bC^N=(Z_0\neq 0)\subset \bP^N$, with $z_i=\frac{Z_i}{Z_0}$ for any $i\in \{1,\dots, N\}$, then $\Eul(dz_i)=\frac{Z_0dZ_i-Z_idZ_0}{Z_0^2}.$  Let us mention that in our computations we will often write $dz_i$ instead of $\Eul(dz_i)$ for simplicity. Those two exact sequences fit together in the following commutative diagram 
\begin{eqnarray}\label{Diagram}
\begin{CD}
  @.            0                @.                 0          @.                                 @.          \\
@.            @VVV                                @VVV                          @.                          @.\\
  @.         \cO_Y(-e)            @=          \cO_Y(-e)        @.                                 @.          \\   
@.            @VVV                                @VVV                          @.                          @.\\ 
0 @>>> \Om_X|_Y @>>>    \tOm_X|_Y     @>>>       \cO_Y  @>>>    0\\  
@.            @VVV                                @VVV                          @|                          @.\\ 
0 @>>>\Om_{Y}  @>>>    \tOm_{Y}      @>>>       \cO_Y   @>>>     0 \\
@.            @VVV                                @VVV                          @.                          @.\\
  @.           0                @.                 0            @.                                 @.         \\ \\
\end{CD}\end{eqnarray}
 
\begin{remark} Observe that $\tOm_X$ can never be ample because it has a trivial quotient. However, Debarre proved in \cite{Deb05} that under the hypothesis of his conjecture, the bundle $\tOm_X(1)$ is ample.
\end{remark}

\subsection{A preliminary example}
The combinatorics needed in the proof of the main results of Section \ref{MainSection} may seem elaborate, but the idea behind it is absolutely elementary. In fact the proofs of the statements in Section \ref{MainSection} are only a repeated us of long exact sequences in cohomology associated to short some exact sequences which are deduced from the restriction exact sequence, the conormal exact sequence, the tilde conormal exact sequence and the Euler exact sequence. But because our purpose is to study tensor produces of symmetric powers of some vector bundle, many indices have to be taken into account, the only purpose of all the notation we will introduce is to synthesis this as smoothly as possible.\\

Let us illustrate the  idea behind this by considering a basic example. Suppose that $H$ is a smooth degree $e$ hypersurface in $\bP^N$ defined by some homogenous polynomial $F\in \bC[Z_0,\dots,Z_N].$ Suppose that we want to understand the groups $H^i(X,S^m\tOm_X(-a))$ for some $a\in \bN,$ and $m\gs 1.$ To do so we look at the tilde conormal exact sequence 
$$0\to \cO_X(-e)\stackrel{\cdot dF}{\to} \tOm_{\bP^N}|_X\to \tOm_X\to 0$$
and take the $m$-th symmetric power and twist it by $\cO_X(-a)$ of it to get the exact sequence
$$0\to S^{m-1}\tOm_{\bP^N}|_X(-e-a)\stackrel{\cdot dF}{\to} S^m\tOm_{\bP^N}|_X(-a)\to S^m\tOm_X(-a)\to 0.$$
By considering the long exact sequence in cohomology associated to it, we see that the groups $H^i(X,S^m\tOm_X(-a))$ can be understood from the groups $H^i(X,S^k\tOm_{\bP^N}|_X(-b))$ for $k,b\in \bN,$ and from the applications appearing in the long exact sequence in cohomology. But to understand those groups, we consider the restriction exact sequence 
$$0\to \cO_{\bP^N}(-e)\stackrel{\cdot F}{\to} \cO_{\bP^N}\to \cO_X\to 0,$$
and twist it by $S^k\tOm_{\bP^N}(-b)$ to get 
$$0\to S^k\tOm_{\bP^N}(-b-e)\stackrel{\cdot F}{\to} S^k\tOm_{\bP^N}(-b)\to S^k\tOm_{\bP^N}|_X(-b)\to 0.$$
Once again, we look at what happens in cohomology, and we see that the groups $H^i(X,S^k\tOm_{\bP^N}|_X(-b))$ can be understood from the groups $H^{i}(\bP^N,S^\ell\tOm_{\bP^N}(-c))$ for $\ell,c\in \bN$ and the maps appearing in the long exact sequence. But observe that $H^{i}(\bP^N,S^\ell\tOm_{\bP^N}(-c))\cong S^\ell\bC^{N+1}\otimes H^i(\bP^N,\cO_{\bP^N}(-c-\ell))=0$ for all $i<N,$ from this we get that   $H^i(X,S^k\tOm_{\bP^N}|_X(-b))=0$ for all $i<N-1$ and from this we deduce that  $H^i(X,S^m\tOm_X(-a))=0$ for all $i<N-2$. Moreover, a more careful study shows that we obtain the following chain of inclusions:
$$H^{N-2}(X,S^m\tOm_X(-a))\stackrel{\tva_1}{\hookrightarrow} H^{N-1}(X,S^{m-1}\tOm_{\bP^N}|_X(-a-e))\stackrel{\tva_2}{\hookrightarrow} H^N(\bP^N,S^{m-1}\tOm_{\bP^N}(-a-2e)).$$
The inclusions appearing in Theorem \ref{MainTheorem} are of this type. Now if one wants to describe what is the image of this composed inclusion, one needs to look more carefully at what are exactly the maps between the cohomology groups in the different long exact sequences. For instance the second injection comes the following exact sequence 
$$0\to H^{N-1}(X,S^{m-1}\tOm_{\bP^N}|_X(-a-e))\to H^N(\bP^N,S^{m-1}\tOm_{\bP^N}(-a-2e))\stackrel{\cdot F}{\to} H^N(\bP^N,S^{m-1}\tOm_{\bP^N}(-a-e)).$$
Hence, $\im(\tva_2)= \ker\left(H^N(\bP^N,S^{m-1}\tOm_{\bP^N}(-a-2e))\stackrel{\cdot F}{\to} H^N(\bP^N,S^{m-1}\tOm_{\bP^N}(-a-e))\right).$
To understand similarly $\im(\tva_2\circ \tva_1)$ is less straightforward, combining the different long exact sequences one obtains the following commutative diagram:
$$
\xymatrix{
H^{N-2}(X,S^m\tOm_X(-a)) \ar[r]^{\!\!\!\!\!\! \tva_1} \ar[dr]^{\varphi} &H^{N-1}(X,S^{m-1}\tOm_{\bP^N}|_X(-a-e)) \ar[r]^{\cdot dF}\ar[d]^{\tva_2} &H^{N-1}(X,S^{m}\tOm_{\bP^N}|_X(-a))\ar[d]\\
& H^N(\bP^N,S^{m-1}\tOm_{\bP^N}(-a-2e)) \ar[r]^{\cdot dF}& H^N(\bP^N,S^{m}\tOm_{\bP^N}(-a-e))
}
$$
where the vertical arrows are injective. Then, by linear algebra, we obtain that $\im (\tva)=\im(\tva_2)\cap \ker(\cdot dF)=\ker(\cdot F)\cap \ker(\cdot dF)$ for suitable maps $\cdot F$ and $\cdot dF.$ This example already contains the main idea of the proof of the first part of Theorem \ref{MainTheorem}. To study more generally tensor products of symmetric powers of the tilde cotangent bundle is done similarly by considering each factor independently, and to deduce the  results concerning the cotangent bundle instead of the tilde cotangent bundle is done in a similar fashion using the Euler exact sequence.

\subsection{An exact sequence}\label{SectionExactSequence}
In the rest of Section \ref{MainSection}, the setting will be the following.
Take an integer $N\geqslant 2$, let $ c \in \{0,\dots,N-1\}$ and take   $e_1,\dots,e_c\in \bN^*$. Take non-zero elements $F_1\in H^0(\bP^N,\cO_{\bP^N}(e_1)),\dots,F_c\in H^0(\bP^N,\cO_{\bP^N}(e_c))$, and for any $i\in \{1,\dots,c\}$ we set $H_i=(F_i=0)$.  For any $i\in \{1,\dots,c\}$ let $X_i:=H_1\cap \cdots \cap H_i.$ Set $X:=X_c$ and $X_0:=\bP^N.$ 
We suppose that $X$ is smooth. For simplicity, we will also suppose that $X_i$ is smooth for each $i\in \{0,\dots,c\}.$
\begin{remark} We make this additional smoothness hypothesis here so that we can work without worrying with all the conormal exact sequences between $X_i$ and $X_{i+1}$ (and this hypothesis will be satisfied in all our applications). However, a more careful analysis of the proof of the main results shows that the only thing we need to have is the smoothness of each of the $X_i'$s in a neighborhood of $X$, and this follows from the smoothness of $X$.\end{remark}


To simplify our exposition, we introduce more notation. If $E$ is a vector bundle on a variety $Y$, and if $\lambda=(\lambda_1,\dots,\lambda_k)$ is a $k$-uple of non-negative integers, then we set 
$$E^{\lambda}:=S^{\lambda_1}E\otimes\cdots\otimes S^{\lambda_k}E.$$
If $\mu=(\mu_1,\dots,\mu_{j})$ is a $j$-uple of non-negative integers, we set 
$$\lambda\cup \mu:=(\lambda_1,\dots,\lambda_k,\mu_1,\dots,\mu_j).$$
The following definition gives a convenient framework for our problem.
\begin{definition}
With the above notation.
\begin{enumerate}
\item A $\lambda$-setting is a $(p+2)$-uple $\Sigma:=(X_p,\lambda^0,\dots,\lambda^p),$ where $0\leqslant p\leqslant c$,
and for any $0\leqslant j\leqslant p$, $ \lambda^j=(\lambda^j_1,\dots,\lambda^j_{m_j})\in\mathbb{N}^{m_j}$ for some $m_j\in \bN$.
\item If $\Sigma=(X_p,\lambda^0,\dots,\lambda^p)$ is as above, we set:
\begin{itemize}
\item $\codim \Sigma :=p$ and  $\dim \Sigma :=N-p.$
\item If $\lambda^1=\cdots =\lambda^p=0$ set $\deg \Sigma:= e_p$. Otherwise, let $j_0:=\min\{1\leqslant j\leqslant p \ {\rm such \ that }\ \lambda^j\neq 0\}$ and set $\deg \Sigma:=e_{j_0}.$

\end{itemize}
\item Take $\Sigma$ as above. We set:
\begin{eqnarray*}
\Om^{\Sigma}:=\Om_{\bP^N|_{X_p}}^{\lambda^0}\otimes\Om_{{X_1}|_{X_p}}^{\lambda^1}\otimes \cdots\otimes \Om_{X_p}^{\lambda^p}\ \ \text{and} \ \  
\tOm^{\Sigma}:=\tOm_{\bP^N|_{X_p}}^{\lambda^0}\otimes\tOm_{{X_1}|_{X_p}}^{\lambda^1}\otimes \cdots\otimes \tOm_{X_p}^{\lambda^p}.
\end{eqnarray*}
\item For any $a\in \mathbb{Z}$ and any $j\in \mathbb{N}$, we set:
\begin{eqnarray*}
H^j\left(\Om^{\Sigma}(a)\right):=H^j\left(X_p,\Om^{\Sigma}\otimes\cO_{X_p}(a)\right) \ \ \text{and}\ \
H^j\left(\tOm^{\Sigma}(a)\right):=H^j\left(X_p,\tOm^{\Sigma}\otimes\cO_{X_p}(a)\right).
\end{eqnarray*}
 \end{enumerate}
\end{definition} 
We will also need a more general definition which will allow us to work simultaneously with $\Om$ and $\tOm$.

\begin{definition}
\begin{enumerate}
\item A $\lambda$-pair $(\Si,\tSi)$ is a couple of $\lambda$-settings $\Si=(X_p,\lambda^0,\dots,\lambda^p)$ and $\tSi=(X_{\tilde{p}},\tla^0,\dots,\tla^{\tilde{p}})$ such that $p=\tilde{p}.$
\item Given a $\lambda$-pair $(\Si,\tSi)$ we set:
\begin{itemize}
\item $\dim(\Si,\tSi):=\dim \Si=\dim \tSi$ and $\codim (\Si,\tSi):=\codim \Si=\codim \tSi.$
\item If $\la^j=\tla^j=0$ for all $i\in \{1,\dots,p\}$ we set $\deg(\Si,\tSi):=e_p.$ Otherwise, let $j_0:=\min\{j\in \{1,\dots, p\} \ \text{such that} \ \ \la^j\neq0 \ \text{or} \ \tla^j\neq 0\}$ and set $\deg(\Si,\tSi)=e_{j_0}.$

\end{itemize}
\item With the above notation, we set $\Om^{(\Si,\tSi)}:=\Om^{\Si}\otimes\tOm^{\tSi}.$
\item For any $a\in\mathbb{Z}$ and any $j\in \mathbb{N}$, we set $H^j\left(\Om^{(\Si,\tSi)}(a)\right):=H^j\left(X_p,\Om^{(\Si,\tSi)}\otimes\cO_{X_p}(a)\right).$ 
\end{enumerate}
\end{definition}
To describe our fundamental exact sequence, we introduce some notion of successors.
\begin{definition}
Take a $\lambda$-setting $\Si=(X_p,\la^0,\dots,\la^p)$ with $p\geqslant 1$ and where for any $0\leqslant j \leqslant p$ one denotes $\la^j=(\la^j_1,\dots,\la^j_{m_j})$. We define $\lambda$-settings $\s_1(\Si)$ and $\s_2(\Si)$ as follows.
\begin{itemize}
\item If $\la^j=0$   for all  $1\leqslant j \leqslant p$ then set $\s_1(\Si):=\s_2(\Si):=(X_{p-1},\la^0,\dots, \la^{p-1}).$
\item If there exists $1\leqslant j\leqslant p$ such that $\la^j\neq 0$, let $j_0:=\min \{j \ /\ 1\leqslant j \leqslant p \ {\rm and}\ \la^j\neq 0\}$ and let $i_0:=\min \{i \ /\ 1\leqslant i \leqslant m_{j_0}\ {\rm and}\ \la^{j_0}_i\neq 0\}$. Then we define 
\begin{eqnarray*}
\s_1(\Si)&:=&(X_p,\la^0,0,\dots,0,(\la^{j_0}_{i_0}),(\la^{j_0}_{i_0+1},\dots,\la^{j_0}_{m_{j_0}}),\la^{j_0+1},\dots,\la^p)\\
\s_2(\Si)&:=&(X_p,\la^0,0,\dots,0,(\la^{j_0}_{i_0}-1),(\la^{j_0}_{i_0+1},\dots,\la^{j_0}_{m_{j_0}}),\la^{j_0+1},\dots,\la^p).
\end{eqnarray*}
 \end{itemize}
\end{definition}
We will need the following generalization  to $\lambda$-pairs.
\begin{definition}
Take a $\lambda$-pair $(\Si,\tSi)$ where $\Si=(X_p,\la^0,\dots,\la^p)$ and $\tSi=(X_p,\tla^0,\dots,\tla^p)$. Define $\s_1(\Si,\tSi)$ and $\s_2(\Si,\tSi)$ as follows.
\begin{itemize}
\item If $\la^j=\tla^j=0$ for all $1\leqslant i \leqslant p$ set $\s_1(\Si,\tSi):=\s_2(\Si,\tSi):=(\s_2(\Si),\s_2(\tSi))=(\s_1(\Si),\s_1(\tSi)).$
\item If there exists $1\leqslant j\leqslant p$ such that $\la^j\neq 0$ or $\tla^j\neq 0$ set $j_0:=\min \{j \geqslant 1\ /\ \la^j\neq 0 \ {\rm or}\ \tla^j\neq 0 \}$.
\begin{itemize}
\item If $\tla^{j_0}\neq 0$ set $\s_1(\Si,\tSi)=(\Si,\s_1(\tSi))\ {\rm and}\ \s_2(\Si,\tSi)=(\Si,\s_2(\tSi)).$
\item If $\tla^{j_0}= 0$ set $\s_1(\Si,\tSi)=(\s_1(\Si),\tSi)\ {\rm and}\ \s_2(\Si,\tSi)=(\s_2(\Si),\tSi).$
\end{itemize}
\end{itemize}
\end{definition}
Now we come to an elementary, but important, observation.
\begin{proposition}\label{Suite-Exacte}
For any $\la$-pair $(\Si,\tSi)$, we have a short exact sequence
\begin{eqnarray}
0\to \Om^{\s_2(\Si,\tSi)}(-\deg(\Si,\tSi))\to \Om^{\s_1(\Si,\tSi)}\to \Om^{(\Si,\tSi)}\to 0 \label{ES1}
\end{eqnarray}
\end{proposition}
\begin{proof}
Take $\Si=(X_p,\la^0,\dots,\la^p)$ and $\tSi=(X_p,\tla^0,\dots,\tla^p).$ We have to consider two cases. \\

\emph{ Case 1: $\la^1=\cdots=\la^p=\tla^1=\cdots=\tla^p=0.$} Set $\Si':=\s_1(\Si)=\s_2(\Si)$ and $\tSi':=\s_1(\tSi)=\s_2(\tSi)$ so that $\tOm^{(\Si',\tSi')}_{|_{X_p}}=\tOm^{(\Si,\tSi)}.$ We have the restriction exact sequence
$$0\to \cO_{X_{p-1}}(-e_p)\to \cO_{X_{p-1}}\to \cO_{X_p}\to 0.$$
Since $e_p=\deg(\Si,\tSi)$, it suffices to tensor this exact sequence by $\tOm^{(\Si',\tSi')}$ to obtain the desired exact sequence.\\

\emph{Case 2:  there exists $j\geqslant 1$ such that $\la^j\neq 0$ or $\tla^j\neq 0$.} Set $j_0:=\min \{j\geqslant 1 \ /\ \la^j\neq 0  \ {\rm or}\ \tla^j\neq 0\}$. Suppose in a first time that $\tla^{j_0}\neq 0$. Recall that $e_{j_0}=\deg (\Si,\tSi)$. Set also $i_0:=\min \{i \ /\ 1\leqslant i \leqslant m_{j_0}\ {\rm and}\ \la^{j_0}_i\neq 0\}$.
Set $\tSi':=(X_p,\tla^0,0,\dots,0,(\tla^{j_0}_{i_0+1},\dots,\tla^{j_0}_{m_{j_0}}),\tla^{j_0+1},\dots, \tla^p),$ so that $\Om^{(\Si,\tSi)}=\Om^{\Si}\otimes S^{\tla^{j_0}_{i_0}}\tOm_{X_{j_0}}|_{X_p}\otimes \tOm^{\tSi'}$,  $\Om^{\s_1(\Si,\tSi)}=\Om^{\Si}\otimes S^{\tla^{j_0}_{i_0}}\tOm_{X_{j_0-1}}|_{X_p}\otimes \tOm^{\tSi'}$ and $\Om^{\s_2(\Si,\tSi)}=\Om^{\Si}\otimes S^{\tla^{j_0}_{i_0}-1}\tOm_{X_{j_0-1}}|_{X_p}\otimes \tOm^{\tSi'}$.
By taking the $\tla^{j_0}_{i_0}$-th symmetric power of the tilde conormal exact sequence when $X_{j_0}$ is seen as a hypersurface of $X_{j_0-1}$ and restricting everything to $X_p$, we obtain
$$0\to S^{\tla^{j_0}_{i_0}-1}\tOm_{X_{j_0-1}}(-e_{j_0})|_{X_p}\to S^{\tla^{j_0}_{i_0}}\tOm_{X_{j_0-1}}|_{X_p}\to S^{\tla^{j_0}_{i_0}}\tOm_{X_{j_0}}|_{X_p}\to 0.$$
It suffices now to tensor this exact sequence by $\tOm^{\tSi'}$ and $\Om^{\Si}$ to obtain the desired result. If $\la^{j_0}=0$ with use the same decomposition on $\Si$ (and the $\la^j$'s) and we use the usual conormal exact sequence instead of the tilde conormal exact sequence.
\end{proof}
\subsection{A vanishing lemma}
In this section, we prove a vanishing lemma which we will often use later. To be able to give the statement, we need some more notation.
Given any $m$-uple of integers $\la=(\la_1,\dots,\la_m)$, we define $\nz(\la)=\sharp\{i\ {\rm such \ that} \ \lambda_i\neq 0\}$ to be the number of non-zero terms in $\la.$ 
\begin{definition} Take a $\la$-setting $\Si=(X_p,\la^0,\dots,\la^p)$, where for all $j\in \{0,\dots, p\}$, $\la^j=(\la^j_1,\dots, \la^j_{m_j}).$ Then we set 
\begin{itemize}
\item $q(\Si):=\dim(\Si)-\n(\Si)$, where $\n(\Sigma):=\sum_{j=1}^p\sum_{i=1}^{m_j}\min\{j,\lambda_i^j\}$.
\item $i(\Si):=\codim(\Si)+\w(\Si)$, where  $\w(\Si):=\sum_{j=1}^pj\nz(\la^j)$.
\item  $t(\Si):=|\Si|-\nz(\Si)$, where $|\Sigma|:=\sum_{j=0}^p\sum_{i=1}^{m_j}\lambda^j_i$ and $\nz(\Sigma):=\sum_{j=0}^p\nz(\la^j)$.
\end{itemize}
\end{definition} 
We will also need the generalization to $\lambda$-pairs. 
\begin{definition}
Take a $\la$-pair $(\Si,\tSi)$ where $\Si=(X_p,\la^0,\dots,\la^p)$, $\tSi=(X_p,\tla^0,\dots,\tla^p).$ Then we set 
\begin{itemize}
\item $q(\Si,\tSi):=\dim(\Si,\tSi)-\n(\Si,\tSi)$, where $\n(\Sigma,\tSi):=\n(\Si)+\n(\tSi)$.
\item $i(\Si,\tSi):=\codim(\Si,\tSi)+\w(\Si,\tSi)$, where  $\w(\Si,\tSi):=\w(\Si)+\w(\tSi)$.
\item $t(\Si,\tSi):=|(\Si,\tSi)|-\nz(\Si),$ where $|(\Sigma,\tSi)|=|\Si|+|\tSi|$.
\end{itemize}
\end{definition}
\begin{remark} Let us just mention what the purpose of these functions are. The integer $q(\Si,\tSi)$ will be the degree of the cohomology group on which we will get some description. The integer $i(\Si,\tSi)$ will be used as a counter in several induction arguments, and $t(\Si,\tSi)$ (as well as $|\Si|$) will be a bound on the twist by $\cO_X(a)$ we can allow in the statements of our results.
\end{remark}
It is straightforward but crucial to observe how those notions behave with respect to the successors introduces in Section \ref{SectionExactSequence}.
\begin{proposition}\label{Prop-elementaire}
For any $\la$-setting $\Si_0$ and any $\la$-pair $(\Si,\tSi)$  we have:
\begin{enumerate}
\item  $q(\s_1(\Si_0))\geqslant q(\Si_0)$ and $q(\s_1(\Si,\tSi))\geqslant q(\Si,\tSi).$ 
\item $q(\s_2(\Si_0))=q(\Si_0)+1$ and $q(\s_2(\Si,\tSi))= q(\Si,\tSi)+1.$ 
\item $i(\s_1(\Si_0))<i(\Si_0)$ and $i(\s_1(\Si,\tSi))<i(\Si,\tSi).$ 
 \item  $i(\s_2(\Si_0))<i(\Si_0)$ and $i(\s_2(\Si,\tSi))<i(\Si,\tSi).$ 
 \item $|\s_1(\Si_0)| =|\Si_0|$ and $t(\s_1(\Si,\tSi))=t(\Si,\tSi).$
  \item $|\s_2(\Si_0)|=|\Si_0|-1$ and $t(\s_2(\Si,\tSi))\gs t(\Si,\tSi)-1.$
\end{enumerate}
\end{proposition}
We are now in position to state and prove our vanishing result.
\begin{lemma}\label{VanishingLemma}
Take a $\lambda$-pair $(\Si,\tSi)$. Take $a\in\mathbb{Z}$ and $j\in \mathbb{N}$. If $j<q(\Si,\tSi)$ and $a<t(\Si,\tSi)$, then
$$H^j\left(\Om^{(\Si,\tSi)}(a)\right)=0.$$
\end{lemma}

\begin{remark}
If we specialize this lemma to $\tSi=(X,0,\dots,0)$, $\Si=(X,0,\dots,0,(\ell_1,\dots,\ell_k))$ and $a=0$, we obtain that $H^j(X,S^{\ell_1}\Om_X\otimes \cdots \otimes S^{\ell_k}\Om_X)=0$ if $j<\dim X-\sum_{i=1}^k\min\{c,\ell_i\}.$ This is the same conclusion as in Br\"uckmann and Rackwitz theorem. Therefore Lemma \ref{VanishingLemma} can be seen as a generalization of their result. 
\end{remark}



\begin{proof}[Proof of Lemma \ref{VanishingLemma}]
We make an induction on $i(\Si,\tSi).$ \\

\emph{If $i(\Si,\tSi)=0$}: Then $\Si=(\bP^N,\la^0)$ and $\tSi=(\bP^N,\tla^0)$, it is a straightforward induction on $\nz(\la^0)$ using the symmetric powers of the Euler exact sequence.\\

\emph{If $i(\Si,\tSi)>0$}:  By Proposition \ref{Prop-elementaire} we can apply our induction hypothesis to $\s_1(\Si,\tSi)$ and $\s_2(\Si,\tSi)$. Observe also that if $a<t(\Si,\tSi)$ then by Proposition \ref{Prop-elementaire} we get $a<t(\s_1(\Si,\tSi))$ and also $a-\deg (\Si,\tSi)<t(\s_2(\Si,\tSi))$. We apply our induction hypothesis and Proposition \ref{Prop-elementaire} to obtain the vanishings
\begin{eqnarray}
H^j\left(\Om^{\s_1(\Si,\tSi)}(a)\right)=0\ &{\rm for }&\ j< q(\Si,\tSi)\leqslant q(\s_1(\Si,\tSi)),\label{V1}\\
H^j\left(\Om^{\s_2(\Si,\tSi)}(a-\deg (\Si,\tSi))\right)=0 \ &{\rm for }& \ j< q(\Si,\tSi)+1=q(\s_2(\Si,\tSi)).\label{V2}
 \end{eqnarray}
 By Proposition \ref{Suite-Exacte} we have the following exact sequence 
 \begin{eqnarray}
 0\to \Om^{\s_2(\Si,\tSi)}(a-\deg(\Si,\tSi))\to \Om^{\s_1(\Si,\tSi)}(a)\to \Om^{(\Si,\tSi)}(a)\to 0.\label{SE1}
 \end{eqnarray}
 It then suffices to use \eqref{V1} and \eqref{V2} in the cohomology long exact sequence associated to \eqref{SE1} to obtain the desired result.
\end{proof}

\subsection{Statements for the tilde cotangent bundle }\label{SectionStatements}
In this section we will prove the first half of Theorem \ref{MainTheorem}, to be precise, we will describe the map $\tva$ and give the announced description of $\im(\tva).$ In fact, to prove this result we will need to prove a more general statement. Before we state our  results we need more notation and some more definitions. If $\ell\in \mathbb{N}$ and if $\la:=(\la_1,\dots,\la_k)$ is a $k$-uple of integers such that $\la_i\geqslant \ell$ for any $i\in\{1,\dots,k\}$ then we set 
$$\la-\ell\bI:=(\la_1-\ell,\dots,\la_k-\ell).$$
If we have two  $\la$-settings of the same dimension $\Si_1=(X_p,\la^0,\dots,\la^p)$ and $\Si_2=(X_p,\mu^0,\dots,\mu^p)$, we set 
$$\Si_1\cup\Si_2:=(X_p,\la^0\cup\mu^0,\dots,\la^p\cup\mu^p).$$

\begin{definition}
Take a $\la$-setting $\Si=(X_p,\la^0,\dots,\la^p)$ where for all $0\leqslant j\leqslant p$ we denote $\la^j=(\la^j_1,\dots,\la^j_{m_j}).$ We say that $\Si$ is \emph{simple} if for any $1\leqslant j\leqslant p$ and for any $1\leqslant i\leqslant m_j$, $\la^j_i\geqslant j$.
More generally, we say that a $\la$-pair $(\Si,\tSi)$ is \emph{simple} if $\Si$ and $\tSi$ are simple.
\end{definition}
It is easy to observe how simplicity behaves with respect to the successors $\s_1$ and $\s_2$.
\begin{proposition}\label{qsimplicity}
If $\Si$ is simple, then $\s_1(\Si)$ and $\s_2(\Si)$ are simple and moreover $$q(\s_1(\Si))=q(\s_2(\Si))=q(\Si)+1.$$
\end{proposition}
\begin{definition} If $\Si$ is simple, we define
$$\Si_{\lim}:=(\bP^N,\la^0\cup(\la^1-\bI)\cup (\la^2-2\bI)\cup\cdots\cup (\la^p-p\bI)),$$
and
$$b_{\Si}:=\sum_{i=1}^pe_i\left(1+\sum_{j=i}^pm_j\right)=\sum_{i=1}^pe_i+\sum_{j=1}^pm_j\sum_{i=1}^je_i.$$
\end{definition}

Fix a simple setting $\Si=(X_p,\la^0,\dots,\la^p)$ as above. Note that for any $ i \in \{1,\dots, p\}$ one has a natural morphism 
$$\cO_{\bP^N}\stackrel{\cdot F_i}{\longrightarrow} \cO_{\bP^N}(e_i)$$
which induces a morphism 
$$\tOm^{\Si_{\lim}}(a) \stackrel{\cdot F_i}{\longrightarrow} \tOm^{\Si_{\lim}}(a+e_i).$$
This induces an application in cohomology, 
$$H^N(\tOm^{\Si_{\lim}}(a)) \stackrel{\cdot F_i}{\longrightarrow} H^N(\tOm^{\Si_{\lim}}(a+e_i)),$$
which we will still denote $\cdot F_i$. This should not lead to any confusion.\\

Similarly, taking the symmetric powers of the application \eqref{foisdF}, for any $i \in \{1,\dots, p\}$ and for any $m\in \mathbb{N}$ one has a natural morphism  
$$S^m\tOm_{\bP^N}\stackrel{\cdot dF_i}{\longrightarrow} S^{m+1}\tOm_{\bP^N}(e_i).$$
Now fix $j \in \{1,\dots, p\}$ and $k \in\{1,\dots, m_j\}$. Set 
\begin{eqnarray*}
\eta^1&:=&\la^0\cup(\la^1-\bI)\cup\cdots\cup(\la^{j-1}-(j-1)\bI)\cup(\la^j_1-j,\dots,\la^j_{k-1}-j)\\
\eta^2 &:=& (\la^j_{k+1}-j,\dots,\la^j_{m_j}-j)\cup(\la^{j+1}-(j+1)\bI)\cup\cdots\cup(\la^p-p\bI).
\end{eqnarray*}
So that $\tOm^{\Si_{\lim}}=\tOm_{\bP^N}^{\eta^1}\otimes S^{\la^j_k-j}\tOm_{\bP^N}\otimes \tOm_{\bP^N}^{\eta^2}$. As above, we get a morphism 
$$S^{\la^j_k-j}\tOm_{\bP^N}\stackrel{\cdot dF_i}{\longrightarrow} S^{\la^j_k-j+1}\tOm_{\bP^N}(e_i).$$
If we tensor this morphism by $\tOm_{\bP^N}^{\eta^1}$, $\tOm_{\bP^N}^{\eta^2}$ and $\cO_{\bP^N}(a)$ we get a morphism
$$\tOm^{\Si_{\lim}}(a)\stackrel{\cdot dF_i^{\{j,k\}}}{\longrightarrow} \tOm_{\bP^N}^{\eta^1}\otimes S^{\la^j_k-j+1}\tOm_{\bP^N}\otimes \tOm_{\bP^N}^{\eta^2}(a+e_i),$$
which yields a morphism 
$$H^N\left(\tOm^{\Si_{\lim}}(a)\right)\stackrel{\cdot dF_i^{\{j,k\}}}{\longrightarrow} H^N\left(\bP^N,\tOm_{\bP^N}^{\eta^1}\otimes S^{\la^j_k-j+1}\tOm_{\bP^N}\otimes \tOm_{\bP^N}^{\eta^2}(a+e_i)\right).$$
We are now in position to state our result.
\begin{theorem}\label{LimiteSetting}
Take a simple $\la$-setting $\Si=(X_p,\la^0,\dots,\la^p)$ as above. Take an integer $a<|\Si|$. Then, there exists an injection 
$$H^{q(\Si)}\left(\tOm^{\Si}(a)\right)\stackrel{\tva}{\hookrightarrow} H^N\left(\tOm^{\Si_{\lim}}(a-b_{\Si})\right).$$
Moreover,
$$\tva\left(H^{q(\Si)}\left(\tOm^{\Si}(a)\right)\right)=\bigcap_{i=1}^p \left(\ker(\cdot F_i)\bigcap_{j=i}^p\bigcap_{k=1}^{m_j}\ker\left(\cdot dF_i^{\{j,k\}}\right)\right).$$
\end{theorem}
To simplify our presentation, we will decompose the proof of Theorem \ref{LimiteSetting} in a couple of propositions.
The following proposition describes the map $\tva.$
\begin{proposition}\label{ChainInclusion}
With the notation of Theorem \ref{LimiteSetting}. Let $q:=q(\Si).$ There is a chain of inclusions:
\begin{eqnarray*}
H^q(\tOm^{\Si}(a))&\hookrightarrow& H^{q+1}\left(\tOm^{\s_2(\Si)}(a-\deg \Si)\right)\\
&\hookrightarrow& H^{q+2}\left(\tOm^{\s_2^2(\Si)}(a-\deg \Si-\deg\s_2(\Si))\right)\\
&\cdots&\\
&\hookrightarrow& H^{q+k}\left(\tOm^{\s_2^k(\Si)}\left(a-\sum_{i=0}^{k-1}\deg\s_2^i(\Si)\right)\right)\\
&\cdots&\\
&\hookrightarrow& H^{N}\left(\tOm^{\s_2^{N-q}(\Si)}\left(a-\sum_{i=0}^{N-q-1}\deg\s_2^i(\Si)\right)\right)
\end{eqnarray*}
And moreover
\begin{enumerate}
\item $\s_2^{N-q}(\Si)=\Si_{\lim}$.
\item $\sum_{i=0}^{N-q-1}\deg\s_2^i(\Si)=b_{\Si}.$
\end{enumerate}
\end{proposition}
\begin{proof}
This is an induction on $i(\Si).$ \\

If $i(\Si)=0$, the result is clear since $\Si=(\bP^N,\la^0)$ and therefore $q=N$, $b_{\Si}=0$ and $\Si=\Si_{\lim}.$\\

Now suppose $i(\Si)>0.$ Then we can suppose, by induction, that the result holds for $\s_2(\Si)$ which is simple and satisfies $q(\s_2(\Si))=q(\Si)+1$. Consider the exact sequence
$$0\to \tOm^{\s_2(\Si)}(a-\deg\Si)\to \tOm^{\s_1(\Si)}(a)\to \tOm^{\Si}(a)\to 0.$$
By Proposition \ref{qsimplicity} and Lemma \ref{VanishingLemma} we obtain
$$H^j\left(\tOm^{\s_2(\Si)}(a-\deg\Si)\right)=H^j\left(\tOm^{\s_1(\Si)}(a)\right)=0\ {\rm if} \ j\leqslant q.$$
By looking at the long exact sequence induced in cohomology we obtain an exact sequence
\begin{eqnarray}0\to H^q\left(\tOm^{\Si}(a)\right)\to H^{q+1}\left(\tOm^{\s_2(\Si)}(a-\deg \Si)\right)\to H^{q+1}\left(\tOm^{\s_1(\Si)}(a)\right).\label{Star}\end{eqnarray}
Applying the induction hypothesis to $\s_2(\Si)$ we get the desired chain of inclusions.\\
Now observe that $\s_2(\Si)_{\lim}=\Si_{\lim}$, and therefore, by induction, we obtain $\s_2^{N-q}(\Si)=\Si_{\lim}.$\\
To see the last point, we have to prove that $b_{\Si}=b_{\s_2(\Si)}+\deg\Si.$ If $\la^1=\cdots=\la^p=0$, then $\deg \Si=e_p$, $b_{\Si}=\sum_{i=1}^pe_i$ and $b_{\s_2(\Si)}=\sum_{i=1}^{p-1}e_i,$ and therefore $b_{\Si}=b_{\s_2(\Si)}+\deg\Si.$ Now if there exists $1\leqslant j\leqslant p $ such that $\la^j\neq 0$. Set, as usual, $j_0:=\min \{j \geqslant 1\ /\ \la^j\neq 0 \}$, so that $\deg\Si=e_{j_0}.$ Then $m_j=0$ for all $1\leqslant j< j_0$ and therefore
\begin{eqnarray*}
b_{\s_2(\Si)}=\sum_{i=1}^pe_i+\sum_{j=j_0+1}^pm_j\sum_{i=1}^je_i+\sum_{i=1}^{j_0-1}e_i+(m_{j_0}-1)\sum_{i=1}^{j_0}e_i
= \sum_{i=1}^pe_i+\sum_{j=j_0+1}^pm_j\sum_{i=1}^je_i -e_{j_0}= b_{\Si}-\deg\Si.
\end{eqnarray*}
\end{proof}
The map $$\tva : H^q\left(\tOm^{\Si}(a)\right)\to H^N\left(\tOm^{\Si_{\lim}}(a-b_{\Si})\right)$$
in Theorem \ref{LimiteSetting} is just the composition of all the injections in Proposition \ref{ChainInclusion}.

We will need an easy linear algebra lemma.
\begin{lemma}\label{linalg}
Take a commutative diagram of vector spaces
$$
\xymatrix{
A \ar[r]^g \ar[dr]^{\varphi} &B \ar[r]^{f_1}\ar[d]^{h_1} &C\ar[d]^{h_2}\\
& D\ar[r]^{f_2}& E
}
$$
such that $g,h_1$ and $h_2$ are injective, such that $f_1\circ g=0$ and such that $g(A)=\ker f_1$, then
$$\varphi(A)=h_1(B)\cap \ker f_2.$$ 
\end{lemma}
The key observation is the following proposition.
\begin{proposition}\label{key}
With the notation of Theorem \ref{LimiteSetting}. Let $q:=q(\Si).$ There is a commutative diagram 
$$\xymatrix{
H^q\left(\tOm^{\Si}(a)\right)\ar[r]^{\iota\ \ \ \ \ \ \ \ \ }\ar@/_5pc/[dddddr]^{\tva}&H^{q+1}\left(\tOm^{\s_2(\Si)}(a-\deg \Si)\right)\ar[r]^{\cdot A_q}\ar[d]& H^{q+1}\left(\tOm^{\s_1(\Si)}(a)\right)\ar[d]\\
&\vdots\ar[d] & \vdots\ar[d]\\
&H^{q+k}\left(\tOm^{\s_2^k(\Si)}\left(a-\sum_{i=0}^{k-1}\deg \s_2^i( \Si)\right)\right)\ar[r]^{\cdot A_{q+k}}\ar[d]& H^{q+k}\left(\tOm^{\s_2^{k-1}\s_1(\Si)}\left(a-\sum_{i=1}^{k-1}\deg  \s_2^i(\Si)\right)\right)\ar[d]\\
&H^{q+k+1}\left(\tOm^{\s_2^{k+1}(\Si)}\left(a-\sum_{i=0}^{k} \deg \s_2^i(\Si)\right)\right)\ar[r]^{\cdot A_{q+k+1}}\ar[d]& H^{q+k+1}\left(\tOm^{\s_2^{k}\s_1(\Si)}\left(a-\sum_{i=1}^{k} \deg \s_2^i(\Si)\right)\right)\ar[d]\\
&\vdots\ar[d] & \vdots\ar[d]\\
&H^{N}\left(\tOm^{\s_2^{N-q}(\Si)}\left(a-\sum_{i=0}^{N-q-1} \deg \s_2^i(\Si)\right)\right)\ar[r]^{\cdot A_N}& H^{N}\left(\tOm^{\s_2^{N-q-1}\s_1(\Si)}\left(a-\sum_{i=1}^{N-q-1} \deg \s_2^i(\Si)\right)\right)}$$
where all the vertical arrows and the map $\iota$ are injective. The maps $\cdot A_{q+k}$ are described as follows. Let $\Si=(X_p,\la^0,\dots,\la^p)$. If $\la^1=\cdots=\la^p=0$ the map $\cdot A_{q+k}$ is just the map $\cdot F_p$ induced by $\cO_{\bP^N}\to \cO_{\bP^N}(e_p).$ If there exists $1\leqslant j\leqslant p$ such that $\la^j\neq 0$, set $j_0:= \min \{j\geqslant 1/ \la^j\neq0 \}$, then the maps $\cdot A_{q+k}$ are maps induced by $\cdot dF_{j_0}$ so that $\cdot A_q$ is the one induced by the map appearing in the exact sequence \eqref{Star} and the map $\cdot A_N$  is just $\cdot dF_{j_0}^{\{j_0,1\}}$.
\end{proposition}
The proof of Theorem \ref{LimiteSetting} now easily follows by induction from Proposition \ref{key} and Lemma \ref{linalg}.


\begin{proof}[Proof of Theorem \ref{LimiteSetting}]
We proceed by induction on $i(\Si)$.  If $i(\Si)=0,$ there is nothing to prove. Now suppose $i(\Si)>0.$ With the same notation as above. Thanks to Proposition \ref{key} we obtain the following commutative diagram, whose vertical arrows are injective.

$$
\xymatrix{
H^q\left(\tOm^{\Si}(a)\right)\ar[r]^{\!\!\!\!\!\!\!\!\!\!\!\!\!\!\!\!\iota }\ar[dr]^{\tva}&H^{q+1}\left(\tOm^{\s_2(\Si)}(a-\deg \Si)\right)\ar[r]^{\cdot A_q}\ar[d]^{\tva_2}& H^{q+1}\left(\tOm^{\s_1(\Si)}(a)\right)\ar[d]\\
&H^{N}\left(\tOm^{\s_2^{N-q}(\Si)}\left(a-b_\Si\right)\right)\ar[r]^{\!\!\!\!\!\!\!\!\!\!\!\!\!\!\!\! \cdot A_N}& H^{N}\left(\tOm^{\s_2^{N-q-1}\s_1(\Si)}\left(a-b_\Si+\deg(\Si)\right)\right)
}
$$

By Lemma \ref{linalg} we obtain $\im \tva =\im \tva_2\cap \ker(\cdot A_N)$. Now it suffices to apply the induction hypothesis to $\s_2(\Si)$ to obtain the announced description of $\im \tva_2$, which  induces the announced description for $\im \tva.$
\end{proof}


We now give the proof of the proposition.
\begin{proof}[Proof of Proposition \ref{key}]
Let $\Si=(X_p,\la^0,\dots,\la^p)$. We need to treat two cases. \\

\emph{Case 1:  $\la^1=\cdots=\la^p=0$, then $\s_1(\Si)=\s_2(\Si).$} Therefore one has for any $k\geqslant 0$ a commutative diagram 

$$
\xymatrix{
0\ar[d] & 0 \ar[d]\\
\tOm^{\s_2^{k+1}(\Si)}\left(a-\sum_{i=0}^k\deg \s_2^{i}(\Si)\right)\ar[d]\ar[r]^{\cdot F_p} & \tOm^{\s_2^{k}\s_1(\Si)}\left(a-\sum_{i=1}^k\deg \s_2^{i}(\Si)\right)\ar[d]\\
\tOm^{\s_1\s_2^{k}(\Si)}\left(a-\sum_{i=0}^{k-1}\deg \s_2^{i}(\Si)\right)\ar[d]\ar[r]^{\cdot F_p} & \tOm^{\s_1\s_2^{k-1}\s_1(\Si)}\left(a-\sum_{i=1}^{k-1}\deg \s_2^{i}(\Si)\right)\ar[d]\\
\tOm^{\s_2^{k}(\Si)}\left(a-\sum_{i=0}^{k-1}\deg \s_2^{i}(\Si)\right)\ar[d]\ar[r]^{\cdot F_p} & \tOm^{\s_2^{k-1}\s_1(\Si)}\left(a-\sum_{i=1}^{k-1}\deg \s_2^{i}(\Si)\right)\ar[d]\\
0&0
}
$$
where the horizontal maps are just multiplication by $F_p$. Here the vertical exact sequences come from the exact sequence \eqref{Star} applied to $s_2^k(\Si)$. Since $\Si$ is simple, we see that $\s^k_2(\Si)$, $\s_2^{k-1}\s_1(\Si)$, $\s_2^{k+1}(\Si)$, $\s_2^{k}\s_1(\Si)$, $\s_1\s_2^{k}(\Si)$ and $\s_1\s_2^{k-1}\s_1(\Si)$ are simple, 
and that $$q(\s^k_2(\Si))=q(\s_2^{k-1}\s_1(\Si))=q(\Si)+k\ \ {\rm and }\ \ q(\s_2^{k+1}(\Si))=q(\s_2^{k}\s_1(\Si))=q(\s_1\s_2^{k}(\Si))=q(\s_1\s_2^{k-1}\s_1(\Si))=q(\Si)+k+1.$$
By considering the diagram in cohomology associated to the above diagram and by applying Lemma \ref{VanishingLemma}, we obtain the following commutative diagram 
$$
\xymatrix{
0\ar[d] & 0 \ar[d]\\
H^{q+k}\left(\tOm^{\s_2^{k}(\Si)}\left(a-\sum_{i=0}^{k-1}\deg \s_2^{i}(\Si)\right)\right)\ar[d]\ar[r]^{\cdot F_p} &H^{q+k}\left( \tOm^{\s_2^{k-1}\s_1(\Si)}\left(a-\sum_{i=1}^{k-1}\deg \s_2^{i}(\Si)\right)\right)\ar[d]\\
H^{q+k+1}\left(\tOm^{\s_2^{k+1}(\Si)}\left(a-\sum_{i=0}^k\deg \s_2^{i}(\Si)\right)\ar[r]^{\cdot F_p}\right) &H^{q+k+1}\left( \tOm^{\s_2^{k}\s_1(\Si)}\left(a-\sum_{i=1}^k\deg \s_2^{i}(\Si)\right)\right)\\
}
$$
now we put all those squares together to obtain our claim.\\

\emph{Case 2: there exists $1\leqslant j\leqslant p$ such that $\la^j\neq0.$ Set $j_0:=\min\{j\geqslant / \la^j\neq 0\}.$} First let us explain what the maps $\cdot A_{q+k}$ are. Observe that for any $1\leqslant k\leqslant N-q,$ the bundles $\tOm^{\s_2^k(\Si)}$ and $\tOm^{\s_2^{k-1}\s_1(\Si)}$ are of the form
\begin{eqnarray*}
\tOm^{\s_2^k(\Si)}&=& \tOm^{\Si_1}\otimes S^{\ell}\tOm_{X_i}\otimes \tOm^{\Si_2}\\
\tOm^{\s_2^{k-1}\s_1(\Si)}&=& \tOm^{\Si_1}\otimes S^{\ell+1}\tOm_{X_i}\otimes \tOm^{\Si_2}
\end{eqnarray*}
where $\Si_1$ and $\Si_2$ are simple. The map $\cdot dF_{j_0}$ is then the map induced by $\Id_{\Si_1}\otimes \cdot dF_{j_0}\otimes \Id_{\Si_2}$. As before we have a commutative diagram
$$
\xymatrix{
0\ar[d] & 0 \ar[d]\\
\tOm^{\s_2^{k+1}(\Si)}\left(a-\sum_{i=0}^k\deg \s_2^{i}(\Si)\right)\ar[d]^{\cdot B}\ar[r]^{\cdot dF_{j_0}} & \tOm^{\s_2^{k}\s_1(\Si)}\left(a-\sum_{i=1}^k\deg \s_2^{i}(\Si)\right)\ar[d]^{\cdot B}\\
\tOm^{\s_1\s_2^{k}(\Si)}\left(a-\sum_{i=0}^{k-1}\deg \s_2^{i}(\Si)\right)\ar[d]\ar[r]^{\cdot dF_{j_0}} & \tOm^{\s_1\s_2^{k-1}\s_1(\Si)}\left(a-\sum_{i=1}^{k-1}\deg \s_2^{i}(\Si)\right)\ar[d]\\
\tOm^{\s_2^{k}(\Si)}\left(a-\sum_{i=0}^{k-1}\deg \s_2^{i}(\Si)\right)\ar[d]\ar[r]^{\cdot dF_{j_0}} & \tOm^{\s_2^{k-1}\s_1(\Si)}\left(a-\sum_{i=1}^{k-1}\deg \s_2^{i}(\Si)\right)\ar[d]\\
0&0
}
$$
where the map $\cdot B$ comes from the short exact sequence \eqref{Star} applied to $s_2^k(\Si)$. More precisely, the map $\cdot B$ will be either induced by $\cdot F_j$ for some $j$, either induced by $\cdot dF_j$ for some $j$. If $\cdot B=\cdot F_j$ then the commutativity is clear. If $\cdot B$ is induced by some map $\cdot dF_j$ then there are two cases to consider. Recall that by construction each of those maps will be the identity on all except   one term in the tensor product. So either the maps $\cdot B$ and $\cdot A$ act on the same factor of the tensor product, either they act on different factors. If they act on different factors, the commutativity is clear. If they act on the same factor it suffices to use the commutativity of the following diagram:
$$
\xymatrix{
0\ar[r] & S^{\ell-1}\tOm_{{X_{i-1}|}_{X_i}}(-2e_i)\ar[d]^{\cdot dF_{j_0}}\ar[r]^{\cdot dF_{i}} &  S^{\ell}\tOm_{{X_{i-1}|}_{X_i}}(-e_i)\ar[d]^{\cdot dF_{j_0}}\ar[r]& S^{\ell}\tOm_{X_i}(-e_i)\ar[d]^{\cdot dF_{j_0}}\ar[r]& 0\\
0\ar[r] & S^{\ell}\tOm_{{X_{i-1}|}_{X_i}}(-e_i)\ar[r]^{\cdot dF_{i}} &  S^{\ell+1}\tOm_{{X_{i-1}|}_{X_i}}\ar[r]& S^{\ell+1}\tOm_{X_i}\ar[r]& 0.
}
$$
Now the rest of the proof follows as in the first case by looking at what happens in cohomology.
\end{proof}
\subsection{Twisting the Euler exact sequence}
From the previous section, we have a good understanding of the groups $H^{q(\Si)}\left(\tOm^{\Si}(a)\right)$ when $\Si$ is simple. Now we want to use this to deduce a similar description for the groups $H^{q(\Si)}\left(\Om^{\Si}(a)\right)$. To do this, we will use cohomological technics similar to the ones we used in the previous sections, but instead of building everything on the restriction exact sequence and the conormal exact sequence, we will use the Euler exact sequence. Again, everything will be based on a suitable exact sequence, and to define it, we need another way of taking successors for simple pairs. 
\begin{definition} Take a simple $\la$-pair $(\Si,\tSi)$ such that $\Si:=(X_p,\la^0,\dots, \la^p)$, $\tSi:=(X_p,\tla^0,\dots, \tla^p)$ and such that $\nz(\Si)\neq 0$. Set $j_0:=\min\{j\geqslant 0\ /\ \la^j\neq 0\}$, $i_0:=\min\{i\gs 1\ /\ \la^{j_0}_i\neq 0\}$ and let 
\begin{eqnarray*}
\Si'&:=&(X_p,0,\dots, 0,(\la^{j_0}_{i_0+1},\dots, \la^{j_0}_{m_{j_0}}),\la^{j_0+1},\dots, \la^p)\\
\tSi'&:=&(X_p,\tla^0,\dots,\tla^{j_0-1},\tla^{j_0}\cup(\la^{j_0}_{i_0}),\tla^{j_0+1},\dots,\tla^{p})\\
\tSi''&=&(X_p,\tla^0,\dots, \tla^{j_0-1},\tla^{j_0}\cup (\la^{j_0}_{i_0}-1),\tla^{j_0+1},\dots, \tla^p),
\end{eqnarray*}
and set 
$$h_1(\Si,\tSi):=(\Si',\tSi') \ \ \ {\rm and} \ \ \ h_2(\Si,\tSi):=(\Si',\tSi'').$$
\end{definition}

We make an elementary observation.
\begin{proposition} If $(\Si,\tSi)$ is a simple pair such that $\nz(\Si)\neq 0$, then $h_1(\Si,\tSi)$ is simple, $q(h_1(\Si,\tSi))=q(\Si,\tSi)$, and $q(h_2(\Si,\tSi))\geqslant q(\Si,\tSi)$. 
\end{proposition}  

The  following proposition is crucial to us.
\begin{proposition}\label{PropositionEuler}
With the above notation.
\begin{enumerate}
\item The Euler exact sequence \eqref{EulerExactSequence} yields an exact sequence
$$0\to \Om^{(\Si,\tSi)}\stackrel{\Eul}{\to} \Om^{h_1(\Si,\tSi)}\to  \Om^{h_2(\Si,\tSi)}\to 0.$$
\item  Suppose $\codim(\Si,\tSi)>0$. For any $1\leqslant i,j\leqslant 2,$ there is an isomorphism 
$$\Om^{h_i \s_j(\Si,\tSi)}\cong \Om^{\s_j h_i(\Si,\tSi)}.$$ 
\item Suppose $\codim(\Si,\tSi)>0$.  Set $b:=\deg (\Si,\tSi)$. There is a commutative diagram 
\begin{eqnarray}\label{GrosDiagram}
\begin{CD}
  @.            0                @.                 0          @.                        0         @.          \\
@.            @VVV                                @VVV                          @VVV                          @.\\
0  @>>>         \Om^{\s_2(\Si,\tSi)}(-b)             @>>>              \Om^{\s_2 h_1(\Si,\tSi)}(-b)    @>>>   \Om^{\s_2h_2 (\Si,\tSi)}(-b)          @>>> 0          \\   
@.            @VVV                                @VVV                          @VVV                         @.\\ 
0 @>>> \Om^{\s_1(\Si,\tSi)} @>>>     \Om^{\s_1 h_1(\Si,\tSi)}     @>>>      \Om^{\s_1 h_2 (\Si,\tSi)}  @>>>    0\\  
@.            @VVV                                @VVV                          @VVV                          @.\\ 
0 @>>>\Om^{(\Si,\tSi)}  @>>>    \Om^{ h_1(\Si,\tSi)}     @>>>       \Om^{h_2 (\Si,\tSi)}   @>>>     0 \\
@.            @VVV                                @VVV                          @VVV                          @.\\
  @.           0                @.                 0            @.                0                 @.         \\ \\
\end{CD}
\end{eqnarray}
\end{enumerate}
\end{proposition}
\begin{proof}
The proof is very similar to what was done in the previous sections, and we give only a rough outline of it. \begin{enumerate}
\item With the notation of the definition of $h_1$ and $h_2$, it suffices to take the $\la_{i_0}^{j_0}$'s power of the Euler exact sequence to get 
$$0\to S^{\la_{i_0}^{j_0}}\Om_{X_{j_0}}\to S^{\la_{i_0}^{j_0}}\tOm_{X_{j_0}}\to S^{\la_{i_0}^{j_0}-1}\tOm_{X_{j_0}}\to 0.$$
Then one just has to tensor this by a suitable tensor product of symmetric powers of $\Om$'s and $\tOm$'s.
\item One just has to check the different cases. Most of the time the announced isomorphism is just the identity, but in some cases, the isomorphism is a reordering of some factors in the tensor product under consideration.
\item One takes the symmetric powers of Diagram \eqref{Diagram} for suitable $X$ and $Y$, and twist it by a suitable tensor product of symmetric powers of $\Om'$s and $\tOm'$s. 
\end{enumerate}
\end{proof}
\subsection{Statements for the cotangent bundle}
We are now in position to state and prove the second half of Theorem \ref{MainTheorem}.
Observe that for any simple $\lambda$-setting $(\Si,\tSi)$, the Euler exact sequence \eqref{EulerExactSequence} yields a morphism 
$\Om^{(\Si,\tSi)}\stackrel{\Eul}{\to} \tOm^{\Si\cup\tSi}.$
Observe also that $\tOm^{\Si_{\lim}\cup\tSi_{\lim}}\cong \tOm^{(\Si\cup\tSi)_{\lim}}$ (just as in Proposition \ref{PropositionEuler}, this isomorphism is just a reordering of the different factors of the tensor product). Therefore, this induces a map $\Eul:\Om^{\Si_{\lim}}\otimes \tOm^{\tSi_{\lim}}\to \Om^{(\Si\cup\tSi)_{\lim}}$.  As usual we will also denote by $\Eul$ the map induced between the different  cohomology groups. We have the following.
\begin{theorem}\label{LimitePair}
If $(\Si,\tSi)$ is a simple pair and $a<t(\Si,\tSi)$ then, there is a commutative diagram 
$$
\xymatrix{
H^{q(\Si,\tSi)}\left(\Om^{(\Si,\tSi)}(a)\right)\ar[r] \ar[dr]^{\varphi} \ar[d]_{\Eul} & H^N\left(\Om^{\Si_{\lim}}\otimes \tOm^{\tSi_{\lim}}(a-b_{\Si}-b_{\tSi})\right)\ar[d]^{\Eul}\\
H^{q(\Si,\tSi)}\left(\tOm^{\Si\cup\tSi}(a)\right)\ar[r]^{\tva\ \ \ \ \ \ \ \ \ } &\ \ \ \  H^N\left(\tOm^{(\Si\cup\tSi)_{\lim}}(a-b_{\Si}-b_{\tSi})\right)
}
$$
where all the arrows are injective and where $\tva$ is the map from Theorem \ref{LimiteSetting}. Moreover, $\im\varphi=\im\tva\cap \im\Eul.$
\end{theorem}

We would like to point out a special case of particular interest (which was denoted by Theorem \ref{MainTheorem} in the introduction).
\begin{corollary}\label{LeCoroQuiTue}
Take integers $\ell_1,\dots,\ell_k\geqslant c$ take an integer $a<\ell_1+\cdots +\ell_k-k$, let $q:=n-kc$ and $b:=(k+1)\sum_{i=1}^ce_i$.  Then one has a commutative diagram 
$$
\xymatrix{
H^q\left(X,\Om_X^{(\ell_1,\dots,\ell_k)}(a)\right)\ar[r] \ar[d]_{\Eul} \ar[dr]^{\varphi} & H^N\left(\bP^N,\Om_{\bP^N}^{(\ell_1-c,\dots,\ell_k-c)}(a-b)\right)\ar[d]^{\Eul}\\
H^q\left(X,\tOm_X^{(\ell_1,\dots,\ell_k)}(a)\right)\ar[r]^{\tva\ \ \ \ \ } & H^N\left(\bP^N,\tOm_{\bP^N}^{(\ell_1-c,\dots,\ell_k-c)}(a-b)\right)\\
}
$$
Such that:
\begin{enumerate}
\item $\im \tva =\bigcap_{i=1}^c \left(\ker(\cdot F_i)\bigcap_{j=1}^k\ker\left(\cdot dF_i^{\{j\}}\right)\right)$. 
\item$ \im \varphi=\im \tva \cap \im\Eul.$
\end{enumerate}
\end{corollary}
In the corollary, the map $\cdot dF_i^{\{j\}}$ is just the map 
$$H^N\left(\bP^N,\tOm_{\bP^N}^{(\ell_1-c,\dots,\ell_k-c)}(a-b)\right)\stackrel{\cdot dF_i^{\{j\}}}{\to}H^N\left(\bP^N,\tOm_{\bP^N}^{(\ell_1-c,\dots\ell_j-c+1,\ell_k-c)}(a-b+e_i)\right)$$
induced by the map 
$$S^{\ell_j-c}\tOm_{\bP^N}\stackrel{\cdot dF_i}{\to}S^{\ell-c+1}\tOm_{\bP^N}(e_i).$$
\begin{remark} Let us consider more precisely the case $k=1$ in Corollary \ref{LeCoroQuiTue}, suppose for simplicity that $N=2N_0$ is even. At first sight, it might seem that if $c>n$ this result doesn't tell us anything, but in fact, we can still get some information by using a simple trick. Indeed, suppose we want to construct symmetric differential forms on $X$, then it suffices to write $Y=H_1\cap \cdots \cap H_{N_0}$ and $X=Y\cap H_{N_0+1}\cap \cdots \cap H_c$. Corollary \ref{LeCoroQuiTue} then gives us information on $H^0(Y,S^m\Om_{Y})$. And if we are able to construct an element  $\om\in H^0(Y,S^m\Om_{Y})$ one can consider the induced restriction $\om_X\in H^0(X,S^m\Om_{X}).$ Similar arguments can by done when $N$ is odd. We would like to mention that in the proof of our main application in Section \ref{SectionAmplitude}, we will in fact use a similar trick with Theorem \ref{LimiteSetting} and not only Corollary \ref{LeCoroQuiTue}.
\end{remark}
\begin{proof}[Proof of Theorem \ref{LimitePair}]
Take a simple $\la$-pair $(\Si,\tSi)$ where $\Si=(X_p,\la^0,\dots,\la^p)$ and $\tSi=(X_p,\tla^0,\dots, \tla^p).$ We make an induction on $\nz(\Si).$ Let $q:=q(\Si,\tSi).$  If $\nz(\Si)=0$ there's nothing to prove. We now suppose that $\nz(\Si)\neq 0$. Diagram \eqref{GrosDiagram} yields a commutative square
$$
\begin{CD}
  H^q\left( \Om^{(\Si,\tSi)}\right)  @>>>   H^q\left( \Om^{ h_1(\Si,\tSi)} \right) \\
          @VVV                                @VVV       \\ 
         H^{q+1}\left(\Om^{\s_2(\Si,\tSi)}(-\deg (\Si,\tSi))\right)           @>>>    H^{q+1}\left( \Om^{\s_2 h_1(\Si,\tSi)}(-\deg (\Si,\tSi))\right)  
\end{CD}
$$
where, by applying Lemma \ref{VanishingLemma}, all the arrows are injective. Now set 
$$H_{j,k}^{q+i}:= H^{q+i}\left(\Om^{s_2^jh_1^k(\Si,\tSi)}\left(a-\sum_{\ell=0}^{j-1}\deg s_2^{\ell}(\Si,\tSi) \right)\right).$$
Putting all the above cartesian squares together, we obtain the following commutative diagram whose arrows are all injective.
\begin{eqnarray}\label{HugeDiagram}
\begin{CD}
H^{q}_{0,0}@>>> H^{q}_{0,1}@>>> \cdots @>>> H^{q}_{0,k}@>>> H^{q}_{0,k+1} @>>> \cdots @>>> H^{q}_{0,\nz(\Si)} \\
@VVV @VVV @. @VVV @VVV @.  @VVV\\
H^{q+1}_{1,0}@>>>H^{q+1}_{1,1} @>>>\cdots @.\vdots @. \vdots @. @.\vdots \\
@VVV @VVV @. @VVV @VVV @.  @.\\
\vdots@. \vdots@. @.\vdots @.\vdots @. @. \\
@VVV @. @. @VVV @VVV @.  @.\\
H^{q+j}_{j,0}@>>>\cdots @. \cdots @>>> H^{q+j}_{j,k} @>>> H^{q+j}_{j,k+1} @>>> \cdots @. \\
@VVV @. @. @VVV @VVV @. @.\\
H^{q+j+1}_{j+1,0}@>>> \cdots @.\cdots @>>> H^{q+j+1}_{j+1,k} @>>> H^{q+j+1}_{j+1,k+1}@>>> \cdots @. \\
@VVV @. @. @VVV @VVV @. @.\\
\vdots @. @. @. \vdots @. \vdots @. @. \vdots \\
@VVV @. @. @. @. @. @VVV \\
H^{N}_{N-q,0}@>>> \cdots @. @.\cdots @. @.\cdots @>>> H^{N}_{N-q,\nz(\Si)} 
\end{CD}
\end{eqnarray}

Observe that 
\begin{eqnarray*}
H^{q}_{0,\nz(\Si)}&\cong& H^q(\tOm^{\Si\cup\tSi}(a))\\
H^N_{N-q,0}&\cong& H^N(\Om^{\Si_{\lim}}\otimes \tOm^{\Si_{\lim}}(a-b_{\Si}-b_{\tSi}))\\
H^N_{N-q,\nz(\Si)}&\cong& H^N(\tOm^{\Si_{\lim}\cup\tSi_{\lim}}(a-b_{\Si}-b_{\tSi})).
\end{eqnarray*}
We would like to point out that to make the proof completely precise, one should take into account the different isomorphisms coming from Proposition \ref{PropositionEuler}.2. But for simplicity we neglect those details here.
Set $\varphi$ to be the map $H^q_{0,0}\to H^N_{N-q,\nz(\Si)}$.
Observe also that the composed map $H^q_{0,\nz(\Si)}\to H^N_{N-q,\nz(\Si)}$ is just the map $\tva$ from Theorem \ref{LimiteSetting}. This gives us the commutative diagram of Theorem \ref{LimitePair}. The proof of the  rest of the statement is similar to the proof of Theorem \ref{LimiteSetting}. When one looks at the long exact sequence associated to Diagram \ref{GrosDiagram}, we get a commutative diagram
$$
\begin{CD}
0 @>>>H^q\left(\Om^{(\Si,\tSi)}(a)\right)  @>>>   H^q\left( \Om^{ h_1(\Si,\tSi)}(a)\right)     @>>>      H^q\left( \Om^{h_2 (\Si,\tSi)}(a)\right)\\
@.            @VVV                                @VVV                          @VVV                         \\
0  @>>>     H^{q+1} \left(   \Om^{\s_2(\Si,\tSi)}(a-b)\right)          @>>>        H^{q+1}\left(\Om^{\s_2\circ h_1(\Si,\tSi)}(a-b)    \right)@>>> H^{q+1}\left(\Om^{\s_2\circ h_2 (\Si,\tSi)}(a-b)\right)           \\   
\end{CD}
$$
where $b:=\deg (\Si,\tSi)$. A quick induction yields a commutative diagram
$$
\xymatrix{
0\ar[r] & H^q\left(\Om^{(\Si,\tSi)}(a)\right) \ar[d] \ar[dr]^{\varphi_1} \ar[r]^{\Eul_q} & H^q\left(\Om^{h_1(\Si,\tSi)}(a)\right)\ar[d]^{\varphi_2}\ar[r]& H^q\left(\Om^{h_2(\Si,\tSi)}(a)\right)\ar[d] \\ 
0\ar[r] & H^N\left(\Om^{\s_2^{N-q}(\Si,\tSi)}(a-b_{\Si,\tSi})\right) \ar[r]^{\Eul_N} & H^N\left(\Om^{\s_2^{N-q}h_1(\Si,\tSi)}(a-b_{\Si,\tSi})\right) \ar[r]^{\tau}&H^N\left(\Om^{\s_2^{N-q}h_2(\Si,\tSi)}(a-b_{\Si,\tSi})\right)
}
$$
Where $b_{\Si,\tSi}=b_{\Si}+b_{\tSi}.$ By Lemma \ref{linalg} we obtain $$\im\varphi_1=\im\varphi_2\cap\ker(\tau)=\im\varphi_2\cap \im \Eul_N.$$
From Diagram \ref{HugeDiagram} one can extract the commutative diagram 
$$
\xymatrix{
H^q_{0,0}\ar[d]\ar[r]\ar[dr]^{\varphi_1} \ar@/^3pc/[drr]^{\varphi}& H^q_{0,1}\ar[d]^{\varphi_2}\ar[r]\ar[dr]^{\varphi'} & H^q_{0,\nz(\Si)}\ar[d]^{\tva}\\
H^N_{N-q,0} \ar[r]^{\Eul_N}\ar@/_2pc/[rr]^{\Eul}& H^{N}_{N-q,1} \ar[r]^{\Eul'}& H^N_{N-q,\nz(\Si)}
}
$$
Our induction hypothesis is that $\im \varphi'=\im \Eul'\cap \im\tva.$ Using the fact that $\im\varphi_1=\im\varphi_2\cap \im \Eul_N,$ we get
\begin{eqnarray*}
\im\varphi = \Eul'(\im\varphi_1)=\Eul'(\im \Eul_N)\cap\Eul'(\im\varphi_2)=\im\Eul\cap\im\varphi'=\im\Eul\cap\im\tva\cap\im\Eul'=\im\Eul\cap\im\tva.
\end{eqnarray*}

\end{proof}


\section{Applications}\label{SectionApplications}

We give different applications of theorems \ref{LimiteSetting} and \ref{LimitePair}. These two statements basically give us a way of computing different cohomology groups on complete intersection varieties by reducing the problem to the computation of the kernel of a linear map depending (in some explicit way) on the defining equations of our complete intersection. In general computing this kernel is a  difficult question because of the dimension of the spaces that are involved. However, in some special cases, one can make those computations, and this gives us some noteworthy conclusions. 

\subsection{Explicit computation in \v{C}ech cohomology}\label{Cech}
In this section, we explain how to make theorems \ref{LimiteSetting} and \ref{LimitePair} explicit via the use of \v{C}ech cohomology. We will use the standard homogenous coordinates $[Z_0:\dots:Z_N]$ on $\bP^N$ and the standard affine subsets $U_i:=(Z_i\neq 0)\subset \bP^N$. Let $\fU:=(U_i)_{0\ls i\ls N}.$
\begin{remark}
In some situations it is more natural to use other open coverings of $\bP^N$ more suitable to the geometric problem under consideration. See for instance Section \ref{Curves} for an illustration of this.
\end{remark}
Recall, see for example \cite{Har77}, that if $a\geqslant N+1$ then 
\begin{eqnarray}
H^N(\bP^N,\cO_{\bP^N}(-a))\cong \check{H}^N(\fU,\cO_{\bP^N}(-a))\cong\bigoplus_{\substack{i_0+\cdots+i_N=a\\ i_0,\dots,i_N\geqslant 1}}\frac{1}{Z_0^{i_0}\cdots Z_N^{i_N}}\cdot \bC
\end{eqnarray}
Recall also that
$$\tOm_{\bP^N}=\bigoplus_{i=0}^N\cO_{\bP^N}(-1)dZ_i= V\otimes \cO_{\bP^N}(-1),$$
where $V=\bigoplus_{i=0}^N\bC\cdot dZ_i$.
Now take integers $\ell_1,\dots, \ell_k\geqslant 0$ and $a<\ell_1+\cdots +\ell_k-N-1$. We have
 \begin{eqnarray*}
 H^N\left(\bP^N,\tOm_{\bP^N}^{(\ell_1,\dots,\ell_k)}(a)\right)&\cong&V^{(\ell_1,\dots,\ell_k)}\otimes H^N\left(\bP^N,\cO_{\bP^N}(a-\ell_1-\cdots-\ell_k)\right)\\
 &\cong& V^{(\ell_1,\dots,\ell_k)}\otimes\!\!\!\! \bigoplus_{\substack{i_0,\dots,i_N\geqslant 1\\ i_0+\cdots+i_N=\ell_1+\cdots+\ell_k-a}}\!\!\!\!\frac{1}{Z_0^{i_0}\cdots Z_N^{i_N}}\cdot \bC.
 \end{eqnarray*}
Therefore an element of  $ H^N\left(\bP^N,\tOm_{\bP^N}^{(\ell_1,\dots,\ell_k)}(a)\right)$ can be thought of as an element of the form 
$$\omega=\sum_{\substack{J_1,\dots,J_k,I\in\mathbb{N}^{N+1}\\ |J_1|=\ell_1,\dots,|J_k|=\ell_k\\ |I|=\ell_1+\cdots+\ell_k-a\\ I\geqslant \bI}}\omega_{J_1,\dots,J_k}^I\frac{dZ^{J_1}\otimes\cdots \otimes dZ^{J_k}}{Z^I}.$$
 Where $\omega_{J_1,\dots,J_k}^I\in\bC$ and where $\bI:=(1,\dots,1)\in\mathbb{N}^{N+1}.$ If $K=(k_0,\dots,k_{N})$ and $J=(j_0,\dots,j_N)$ are both in $\mathbb{N}^{N+1}$ we write $K\geqslant J$ if $k_i\geqslant j_i$ for any $0\leqslant i \leqslant N.$ We also use the standard multi-index notation: if $I=(i_0,\dots,i_N)\in \mathbb{N}^{N+1}$ then $Z^I:=Z_0^{i_0}\cdots Z^{i_N}_N$  and  $dZ^I:=dZ_0^{i_0}\cdots dZ^{i_N}_N$.\\
 
 We now describe explicitly the maps $\cdot F$ and $\cdot dF$. Start with a monomial $Z^M\in H^N(\bP^N,\cO_{\bP^N}(e))$ of degree~$e$, where $M\in\mathbb{N}^{N+1}$. Take $\la=(\ell_1,\dots,\ell_k)$ as above. The multiplication by $Z^M$ induces the map
 $$
 \begin{array}{ccc}
 H^{N}\left(\bP^{N},\tOm_{\bP^N}^{\la}(a)\right) & \to & H^N\left(\bP^{N},\tOm_{\bP^N}^{\la}(a+e)\right)\\
 \omega=\frac{dZ^{J_1}\otimes \cdots\otimes dZ^{J_k}}{Z^I} & \mapsto & \omega\cdot Z^M=\left\{\begin{array}{ccc}\frac{dZ^{J_1}\otimes \cdots\otimes dZ^{J_k}}{Z^{I-M}} &{\rm if} & I+\bI\geqslant M\\ 0& {\rm if}& I+\bI \ngeqslant M.\end{array}\right.
 \end{array}
 $$
 If we take any $F\in H^0(\bP^N,\cO_{\bP^N}(e))$, it suffices to decompose $F$ as a sum of monomials and extend the above description by linearity.\\
 
 Now consider an element $\xi=Z^MdZ_i\in H^0\left(\bP^N,\tOm_{\bP^N}(e)\right).$ Fix $1\leqslant j\leqslant k.$ This induces a map
 
  $$
 \begin{array}{ccc}
 H^{N}\left(\bP^{N},\tOm_{\bP^N}^{(\ell_1,\dots,\ell_k)}(a)\right) & \to & H^N\left(\bP^{N},\tOm_{\bP^N}^{(\ell_1,\dots,\ell_j+1,\dots,\ell_k)}(a+e)\right)\\
 \omega=\frac{dZ^{J_1}\otimes \cdots\otimes dZ^{J_k}}{Z^I} & \mapsto & \omega\cdot \xi^{\{j\}}=\left\{\begin{array}{ccc}\frac{dZ^{J_1}\otimes \cdots\otimes dZ_idZ^{J_j}\otimes\cdots\otimes dZ^{J_k}}{Z^{I-M}} &{\rm if} & I+\bI\geqslant M\\ 0& {\rm if}& I+\bI \ngeqslant M.\end{array}\right.
 \end{array}
 $$
 again, it suffices to extend by linearity to describe the maps $\cdot dF^{\{j\}}$ for any $j$ and  any  $F\in H^N(\bP^N,\cO_{\bP^N}(e))$.
 The maps $\Eul$ are understood very similarly and we don't provide all the details here. However, because in our computations we will have to use the coboundary map in the long exact sequence in cohomology associated to a short exact sequence, we would like to recall how such a coboundary map can be understood. Suppose that $X\subseteq \bP^N$ is a projective variety, that $\fU$ is an open covering of $X$ and that one has the following exact sequence of sheaves on $X$
 $$0\to \cF\stackrel{\varphi}{\to} \cE\stackrel{\rho}{\to} \cG\to 0$$
 This gives us the following maps between the \v{C}ech complexes:
  $$
 \xymatrix{
 \vdots \ar^{\check{d}}[d] &  \vdots \ar^{\check{d}}[d]&  \vdots \ar^{\check{d}}[d] \\
 C^{k}(\mathfrak{U},\cF) \ar^{\check{d}}[d] \ar^{\varphi_k}[r]&C^{k}(\mathfrak{U},\cE) \ar^{\check{d}}[d] \ar^{\rho_k}[r]&C^{k}(\mathfrak{U},\cG) \ar^{\check{d}}[d]\\
 C^{k+1}(\mathfrak{U},\cF) \ar^{\check{d}}[d] \ar^{\varphi_{k+1}}[r]&C^{k+1}(\mathfrak{U},\cE) \ar^{\check{d}}[d] \ar^{\rho_{k+1}}[r]&C^{k+1}(\mathfrak{U},\cG) \ar^{\check{d}}[d]\\
 \vdots&\vdots&\vdots
 }
 $$
Here $\check{d}$ denotes the \v{C}ech differential. The coboundary map $\delta_k:\check{H}^{k}(\fU,\cG)\to \check{H}^{k+1}(\fU,\cF)$ is obtained by applying the snake lemma in the above diagram (we suppose that $\fU$ is sufficiently refined so that one can make this work). The problem we will be facing is the following: suppose that $\delta_k$ is injective and that we have $(\si_{i_0,\dots,i_k})\in Z^{k+1}(\fU,\cF)$ representing a class $\si\in \check{H}^{k+1}(\fU,\cF)$ such that $\si \in \im(\delta_k)$, how to compute a \v{C}ech representative for $ \delta^{-1}_k(\si)$ in $C^{k}(\fU,\cG)$? This is just a diagram chase, and goes as follows: first compute $\vphi_{k+1}(\si_{i_0,\dots,i_k})$, by hypothesis, there exists $(\tau_{i_0,\dots, i_k})\in C^k(\fU,\cE)$ such that $\vphi_{k+1}(\si_{i_0,\dots,i_k})=\check{d}(\tau_{i_0,\dots, i_k}).$ And then $\delta_k^{-1}(\si)$ is just represented by the cocycle $\rho_k(\tau_{i_0,\dots, i_k}).$ For us the maps $\vphi_k$ will be either multiplication by $F$ or multiplication by $dF$ for some polynomial $F$ and $\rho_k$ will just by a restriction map.
\subsection{The case of plane curves}\label{Curves}
Let us just explain how we can use this strategy to construct the \og classical\fg\ differential form on smooth curves in $\bP^2$ of genus greater than $1.$ Let $F\in \bC[Z_0,Z_1,Z_2]$ be a homogenous polynomial of degree $e\gs 3.$ Such that the curve $C:=(F=0)\subseteq \bP^2$ is smooth. For any $i\in \{0,1,2\},$ set $F_i:=\frac{\partial F}{\partial Z_i}$ and $U_i^F:=(F_i\neq 0)$ and $\fU^F:=(U_i^F)_{0\ls i\ls 2}.$ Because $C$ is smooth, $\fU^F$ is an open covering of $\bP^2.$ From Proposition \ref{ChainInclusion} we get a diagram whose arrows are all injective
$$\xymatrix{
H^0(C,\tOm_C) \ar[r]^{\tphi_0}\ar[dr]_{\tphi} & H^1(C,\cO_C(-e)) \ar[d]^{\tphi_1}\\
&H^2(\bP^2,\cO_{\bP^2}(-2e))
}
$$
By Corollary \ref{LeCoroQuiTue} we get that $\im(\tphi)=\ker(\cdot F)\cap \ker(\cdot dF).$
We will work in \v{C}ech cohomology  with respect to $\fU^F$. For any degree $e-3$ polynomial $P\in \bC[Z_0,Z_1,Z_2]$ consider the element $\tom^{P,2}\in H^2(\bP^2,\cO_{\bP^2}(-2e))$ given by the following cocycle:
$$\tom^P_{012}:=\frac{P}{F_0F_1F_2}\in \check{H}^2(\fU^F,\cO_{\bP^2}(-2e)).$$
By Euler's formula we have  $eF=F_0Z_0+F_1Z_1+F_2Z_2.$ Hence, we obtain 
\begin{eqnarray*}
\tom^{P}_{012}\cdot F=\frac{P}{e}\frac{1}{F_0F_1F_2}\cdot (F_0Z_0+F_1Z_1+F_2Z_2)=\frac{P}{e}\left(\frac{Z_0}{F_1F_2}+\frac{Z_1}{F_0F_2}+\frac{Z_2}{F_0F_1}\right)=0\in \check{H}^2(\fU^F,\cO_{\bP^2}(-e))
\end{eqnarray*}
And because $dF=F_0dZ_0+F_1dZ_1+F_2dZ_2$ we obtain similarly 
\begin{eqnarray*}
\tom_{012}^P\cdot dF=P\left(\frac{dZ_0}{F_1F_2}+\frac{dZ_1}{F_0F_2}+\frac{dZ_2}{F_0F_1}\right)=0\in \check{H}^2(\fU^F,\tOm_{\bP^2}(-e))
\end{eqnarray*}
hence, we see that $\tom^{P,2}\in \im(\tphi).$ This already proves the existence of a tilde differential form on $C$ and in fact, Corollary \ref{LeCoroQuiTue} even proves that this tilde differential form will yield a true differential form simply because, with the  notation of Corollary \ref{LeCoroQuiTue}, $\im(\Eul)=H^2(\bP^2,\cO_{\bP^2}(-e)).$ It remains to compute $\tphi^{-1}(\tom^{P,2}),$ to do so, we first compute $\tphi_1^{-1}(\tom^{P,2}).$ This is done as follows:
\begin{eqnarray*}
\tom^P_{012}\cdot F= \frac{P}{e}\left(\frac{Z_0}{F_1F_2}+\frac{Z_1}{F_0F_2}+\frac{Z_2}{F_0F_1}\right)=\tom^P_{12}-\tom^P_{02}+\tom^2_{01}=\check{d}\left((\tom_{ij}^P)_{0\ls i<j\ls 2}\right)
\end{eqnarray*}
Where for each $i<j,$ $\tom^P_{ij}=(-1)^k\frac{P}{e}\frac{Z_k}{F_iF_j}$ for $k\in\{0,1,2\}\setminus \{i,j\}$. Therefore, $\tom^{P,1}:=\tphi^{-1}(\tom^{P,2})$ is represented by the cocycle $(\tom_{ij}^P)_{0\ls i<j\ls 2}.$ Now we compute $\tphi_0^{-1}(\tom^{P,1}).$ Let $i,j\in \{0,1,2\}$ such that $i<j$ and let $k\in\{0,1,2\}\setminus \{i,j\}.$ 
\begin{eqnarray*}
\tom^P_{ij}\cdot dF= (-1)^k\frac{P}{e}\frac{Z_k}{F_iF_j}(F_idZ_i+F_jdZ_j+F_kdZ_k)=(-1)^k\frac{P}{e}\left(\frac{Z_kdZ_i}{F_j}+\frac{Z_kdZ_i}{F_i}+\frac{Z_kF_kdZ_k}{F_iF_j}\right).
\end{eqnarray*}
But here we can use the relation $F=0,$ so that $Z_kF_k=-Z_iF_i-Z_jF_j,$ and therefore
\begin{eqnarray*}
\tom^P_{ij}\cdot dF= (-1)^k\frac{P}{e}\left(\frac{Z_kdZ_i}{F_j}+\frac{Z_kdZ_j}{F_i}-\frac{(Z_iF_i+Z_jF_j)dZ_k}{F_iF_j}\right)=
 (-1)^k\frac{P}{e}\left(\frac{1}{F_j}\left|\begin{array}{cc} Z_k& Z_i\\ dZ_k& dZ_i\end{array}\right|+\frac{1}{F_i}\left|\begin{array}{cc} Z_k& Z_j\\ dZ_k& dZ_j\end{array}\right|\right)
\end{eqnarray*}
For any $i\in \{0,1,2\}$, take $j,k\in \{0,1\}\setminus \{i\}$ such that $j<k$, and set $$\tom_i^P:=(-1)^i\frac{P}{eF_i}\left|\begin{array}{cc} Z_j& Z_k\\ dZ_j& dZ_k\end{array}\right|.$$ With this formula, it is straighforward to check that $(\tom^{P}_{ij}\cdot F)_{0\ls i<j\ls 2}=\check{d}\left((\tom^P_i)_{0\ls i\ls 2}\right),$ hence $\tom^{P,0}:=\tphi^{-1}(\tom^{P,2})=\tphi_0^{-1}(\tom^{P,1})$ is represented by the cocycle $(\tom^{P}_i)_{0\ls i \ls 2}.$ To complete our study, we still have to compute the corresponding element $\om^{P,0}\in H^0(C,\Om_C).$ To do this, we only have to dehomogenize $(\tom^P_i)_{0\ls i\ls 2}.$  We are only going to consider the chart $U_0=(Z_0\neq 0)=\bC^2$ with coordinates $(z_1,z_2)$ where $z_1=\frac{Z_1}{Z_0}$ and $z_2=\frac{Z_2}{Z_0}$. Let $f\in \bC[z_1,z_2]$ (resp. $Q\in \bC[z_1,z_2]$) be the dehomogeneization of $F$ (resp. $P$) with respect to $Z_0$. For $i\in \{1,2\}$ let $f_i:=\frac{\partial f}{\partial z_i}$, observe that $f_i$ is the dehomogeneization of $F_i$ with respect to $Z_0$. We let $U_i'=U_0\cap U_i^{F}=(f_i\neq 0).$ Moreover, for $i\in \{1,2\},$ we have 
$$dz_i=\frac{Z_0dZ_i-Z_idZ_0}{Z_0^2}=\frac{1}{Z_0^2}\left|\begin{array}{cc} Z_0 & Z_i \\ dZ_0 & dZ_i\end{array}\right|.$$
Hence if $i\in \{1,2\}$ and $j\in \{1,2\}\setminus \{i\}$ we obtain
\begin{eqnarray*} \tom^P_i=(-1)^i\frac{P}{eF_i}\left|\begin{array}{cc}Z_0&Z_j\\ dZ_0 & dZ_j\end{array}\right|=(-1)^i\frac{Z_0^{e-3}Q}{eZ_0^{e-1}f_i}Z_0^2dz_i=(-1)^i\frac{1}{e}\frac{Q}{f_i}dz_j.
\end{eqnarray*}
From this we obtain that $\om^{P,0}|_{U_0}\in H^0(U_0\cap C,\Om_C)$ is represented in \v{C}ech cohomology with respect to $(U_1'\cap C,U_2'\cap C)$ by the cocycle $(-Q\frac{dz_2}{f_1},Q\frac{dz_1}{f_2}),$ as expected.

\begin{remark} It is true that this computation is longer than the usual one, however the strategy is a bit different. Indeed, the classical approach consists in finding a differential form locally (on $U_0\cap C$) which is done by a \og guessing\fg\ process, after what one checks that the constructed differential form extends on $C.$ Here we somehow approach the problem the other way around, using Corollary \ref{LeCoroQuiTue}, once we have found the element $\om^P_{012},$ we already know that we have a global differential form on $C$, and the entire computation just comes down to compute a local representative of it, this process is long but of a mechanical nature.  In our opinion this second strategy reduces the importance of the \og guessing \fg\ part and is therefore more suitable for higher dimensional generalizations.
\end{remark}

\subsection{Optimality in Br\"uckmann and Rackwitz theorem}
We are now in position to prove optimality in Theorem \ref{B-RVanishing}. By applying theorems \ref{LimiteSetting} and \ref{LimitePair} to a very particular complete intersection, we will prove the following (so called Theorem \ref{theoB} in the introduction).
\begin{proposition}\label{example}
Let $N\geqslant 2$, let $0\leqslant c<N$. Take integers $\ell_1,\dots,\ell_k\gs1$ and $a<\ell_1+\cdots+\ell_k-k$. Suppose  $0\leqslant q:= N-c-\sum_{i=1}^k\min\{c,\ell_i\}$. Then, there exists a smooth complete intersection variety in $\bP^N$ of codimension $c$, such that 
$$H^q(X,S^{\ell_1}\Om_X\otimes\cdots\otimes S^{\ell_k}\Om_X(a))\neq 0.$$
\end{proposition} 
We will construct an example as follows. Take a $c\times (N+1)$ matrix 
$$A=\left(
\begin{array}{ccc}
a_{10} & \cdots & a _{1N}\\
\vdots& \ddots&\vdots \\
a_{c0}&\cdots& a_{cN}
\end{array}\right),
$$
where the $a_{ij}\in\bC$ are such that for any $ p\in \{1,\dots,c\}$, the $p\times p $ minors of the matrix $A$ are non zero. Fix an integer $e\geqslant 1$. For each $p\in \{1,\dots,c\}$, set
$$F_p:=\sum_{j=0}^Na_{pj}X_j^e \ \ \ {\rm and} \ \ \ X_p:=(F_1=0)\cap \cdots\cap (F_p=0).$$
One easily check that $X_i$ is smooth for any $i\in \{1,\dots, c\}.$ We will show that if we take $e\gg 1$ this example is sufficient to prove Proposition \ref{example}
\subsubsection{The simple case}
We first treat the \emph{simple} case. Take a simple setting $\Si=(X_c,\la^0,\dots,\la^c),$ if for any $0\leqslant j \leqslant c$ one denotes $\la^j=(\la^j_1,\dots,\la^j_{m_j})$ where $\la^j_k\geqslant j$ for any $1\leqslant k\leqslant m_j$. In our situation, one obtains: $\n(\Si)=\sum_{j=1}^cjm_j$, $\nz(\Si)=\sum_{j=1}^cm_j$, $|\Si_{\lim}|=\sum_{j=0}^c\sum_{k=1}^{m_j}(\la^j_k-j)$ and $b_{\Si}=e(c+\n(\Si))$. The statement is the following.
\begin{proposition}\label{SimpleCase}Take $X=X_c$ and $\Si$ as above. Take $a<t(\Si).$ If 
$$e\geqslant \frac{N+1+2|\Si_{\lim}|-a}{q(\Si)+1}\ \ \ \ {\rm then}\ \ \ \ H^{q(\Si)}(\Om^{\Si}(a))\neq 0.$$
\end{proposition}
\begin{proof}
It suffices to apply theorems \ref{LimiteSetting} and \ref{LimitePair} to a non-zero element of the form
$$\omega:=\frac{P}{Z_0^{e-1}\cdots Z_N^{e-1}}\bigotimes_{j=0}^c\bigotimes_{k=1}^{m_j}(Z_0dZ_1-Z_1dZ_0)^{\la^{j}_k-j}\in H^N\left(\bP^N,\tOm_{\bP^N}^{\Si_{\lim}}(a-b_{\Si})\right).$$
Where $P\in\bC[Z_0,\dots,Z_N]$ is a homogenous polynomial of degree $(q(\Si)+1)e+a-N-1-2|\Si_{\lim}|.$
The condition on the degree just insures that such an element exists. Our statement follows at once using theorems \ref{LimiteSetting} and \ref{LimitePair} and the description of Section \ref{Cech} with $F_p=\sum_{j=0}^Na_{pj}Z_j^e$ and $dF_p=e\sum_{j=0}^Na_{pj}Z_j^{e-1}dZ_j$.
\end{proof}

\subsubsection{The general case}
To complete the proof of Proposition \ref{example} we also have to treat the non-simple case. If $\Si$ is not simple, we can not apply directly theorems \ref{LimiteSetting} and \ref{LimitePair}.
We need another type of successors. Take any setting $\Si=(X_p,\la^0,\dots,\la^p)$ where $\la^j=(\la^j_1,\dots,\la^j_{m_j}).$ Suppose, without loss of generality, that for all $1\leqslant j\leqslant p$ and for all $1\leqslant i\leqslant m_j$ we have $\la^j_i\geqslant 1$. And suppose also that $\Si $ is \emph{not} simple. We define the successors $c_1$ and $c_2$ as follows: consider $j_0:=\max\left\{1\leqslant  j\leqslant p \ \ {\rm such \ that}\ \ \exists\  1\leqslant i\leqslant m_j\ \  {\rm with} \ \ \la^j_i<j\right\}$ and
$i_0:=\min\left\{ 1\leqslant i\leqslant m_{j_0} \ \ {\rm such\ that} \ \ \la^{j_0}_i<j_0\right\}.$
Set
\begin{eqnarray*}
c_1(\Si)&:=&\left(X_p,\la^0,\dots,\la^{j_0-2},\la^{j_0-1}\cup(\la^{j_0}_{i_0}),\mu,\la^{j_0+1},\dots,\la^p\right)\\
c_2(\Si)&:=&\left(X_p,\la^0,\dots,\la^{j_0-2},\la^{j_0-1}\cup(\la^{j_0}_{i_0}-1),\mu,\la^{j_0+1},\dots,\la^p\right)
\end{eqnarray*}
where $\mu := \left(\la^{j_0}_1,\dots,\la_{i_0-1}^{j_0},\la_{i_0+1}^{j_0},\dots,\la^{j_0}_{m_{j_0}}\right).$
With those notation set also $\deg'\Si:=e_{j_0}.$ As in Proposition \ref{Suite-Exacte} we obtain an exact sequence
$$0\to \Om^{c_2(\Si)}(-\deg'\Si)\to \Om^{c_1(\Si)}\to \Om^{\Si}\to 0.$$
We can now prove the following.
\begin{proposition}\label{NonSimple}
Let  $\Si$ be any setting, and denote $q:=q(\Si).$ For  any integer $a<t(\Si)$,  there exists a simple setting $\Si'$ such that $q(\Si')=q$, and an injection
$$H^q\left(\Om^{\Si'}(a)\right)\hookrightarrow H^q\left(\Om^{\Si}(a)\right).$$
\end{proposition}

\begin{proof}
Of course, we can suppose that $\Si$ is not simple. 
Observe that $q(c_2(\Si))=q(\Si)+1=q+1.$ Therefore, applying Lemma \ref{VanishingLemma}, we obtain $$H^q\left(\Om^{c_2(\Si)}(a-\deg'\Si)\right)=0.$$
This yields an injection 
$$H^q\left(\Om^{c_1(\Si)}(a)\right)\hookrightarrow H^q\left(\Om^{\Si}(a)\right).$$
It suffices now to observe that after a sufficient number $r$ of iterations, we obtain a simple setting $\Si'=c_1^r(\Si)$ such that $q(\Si')=q.$ By induction we therefore get a chain of injections
$$H^q\left(\Om^{\Si'}(a)\right)=H^q\left(\Om^{c_1^r(\Si)}(a)\right)\hookrightarrow \cdots \hookrightarrow H^q\left(\Om^{c_1(\Si)}(a)\right)\hookrightarrow H^q\left(\Om^{\Si}(a)\right).$$ 
\end{proof}
The proof of Proposition \ref{example} now follows directly from propositions \ref{SimpleCase} and \ref{NonSimple}.


\subsection{Examples for the non-invariance under deformation}
Applying theorems \ref{LimiteSetting} and \ref{LimitePair} to a very particular family of complete intersection surfaces we will prove that the numbers $h^0\left(X,S^m\Om_X\right)$ are not deformation invariant as soon as $m\geqslant 2$.
Our example is the following. 
Take $\alpha:=(\alpha_1,\alpha_2)\in \bC^2$ and $\beta:=(\beta_1,\beta_2)\in \bC^2$. Take $a_0,\dots,a_4\in \bC$ such that $a_i\neq a_j$ if $i\neq j$. Take an integer $e\geqslant 5$ and set $e_1=\lfloor \frac{e}{2}\rfloor$ and $e_2=\lceil \frac{e}{2}\rceil.$ Set 
\begin{eqnarray*}
F_{\alpha}&:=& Z_0^e+Z_1^e+Z_2^e+Z_3^e+Z_4^e+\alpha_1Z_0^{e_1}Z_1^{e_2}+\alpha_2Z_2^{e_1}Z_3^{e_2}\\
G_{\be}&:=& a_0Z_0^e+a_1Z_1^e+a_2Z_2^e+a_3Z_3^e+a_4Z_4^e+\beta_1Z_0^{e_1}Z_1^{e_2}+\beta_2Z_2^{e_1}Z_3^{e_2}.
\end{eqnarray*} 
And set $$X_{\alpha,\beta}=(F_{\alpha}=0)\cap(G_{\beta}=0).$$
\begin{proposition}\label{PropFamille}
With the above notation, we have:
\begin{enumerate}
\item $h^0(X_{0,0},S^2\Om_{X_{0,0}})\neq 0$.
\item For generic $\alpha$ and $\beta$, $h^0(X_{\alpha,\beta},S^2\Om_{X_{\alpha,\beta}})= 0$.
\end{enumerate}
\end{proposition}
\begin{proof} The fact that $h^0(X_{0,0},S^2\Om_{X_{0,0}})\neq 0$ is a very particular case of Proposition \ref{SimpleCase}. The rest of the  proof is a straightforward computation, but we give it for the sake of completeness. Observe first that $$H^0\left(X_{\alpha,\beta},S^2\Om_{X_{\alpha,\beta}}\right)\cong H^0\left(X_{\alpha,\beta},S^2\tOm_{X_{\alpha,\beta}}\right).$$
By Theorem \ref{LimiteSetting} and we obtain an injection 
$H^0\left(X_{\alpha,\beta},S^2\tOm_{X_{\alpha,\beta}}\right)\stackrel{\tva}{\hookrightarrow} H^4\left(\bP^N,\cO_{\bP^N}(-4e)\right)$ such that 
$$\tva\left(H^0\left(X_{\alpha,\beta},S^2\tOm_{X_{\alpha,\beta}}\right)\right)=\ker(\cdot F_{\alpha})\cap\ker(\cdot G_{\beta})\cap\ker(\cdot dF_{\alpha})\cap\ker(\cdot dG_{\beta}).$$
Let $\xi\in H^4\left(\bP^N,\cO_{\bP^N}(-4e)\right)$, observe that $\xi\in\ker\left(\cdot dF_{\alpha}\right)\ \ {\rm if \ and \ only \ if \ \ } \xi\in\bigcap_{i=0}^4\ker\left(\cdot \frac{\partial F_{\alpha}}{\partial Z_i}\right).$
Moreover, since $eF=\sum_{i=0}^4Z_i\frac{\partial F_{\alpha}}{\partial Z_i}$, we see that if $\xi\in\bigcap_{i=0}^4\ker\left(\cdot \frac{\partial F_{\alpha}}{\partial Z_i}\right)$ then $\xi\in \ker (\cdot F_{\alpha}).$ A similar argument for $G_{\beta}$ shows that 
$$\im \tva= \bigcap_{i=0}^4\left(\ker\left(\cdot\frac{\partial F_{\alpha}}{\partial Z_i}\right)\cap\ker\left(\cdot\frac{\partial G_{\beta}}{\partial Z_i}\right)\right).$$
Now we proceed to a standart Gauss algorithm  
$$
\begin{array}{ccl}
\xi \in \ker\left(\cdot\frac{\partial F_{\alpha}}{\partial Z_0}\right)\cap\ker\left(\cdot\frac{\partial G_{\beta}}{\partial Z_0}\right)& \Leftrightarrow  &
\left\{
\begin{array}{lcc}
 \sum_{I}\xi^I\frac{1}{Z^{I}}\cdot \frac{\partial F_{\alpha}}{\partial Z_0}&=&0\\
 \sum_{I}\xi^I\frac{1}{Z^{I}}\cdot \frac{\partial G_{\beta}}{\partial Z_0}&=&0\\
\end{array}
\right.\\
&&\\
&\Leftrightarrow  &
\left\{
\begin{array}{lcc}
 \sum_{I}\xi^I\frac{1}{Z^{I}}\cdot(eZ_0^{e-1}+e_1\alpha_1Z_0^{e_1-1}Z_1^{e_2})&=&0\\
 \sum_{I}\xi^I\frac{1}{Z^{I}}\cdot (ea_0Z_0^{e-1}+e_1\beta_1Z_0^{e_1-1}Z_1^{e_2})&=&0\\
\end{array}
\right.\\
&&\\
&\Leftrightarrow  &
\left\{
\begin{array}{lcc}
 \sum_{I}\xi^I\frac{1}{Z^{I}}\cdot Z_0^{e-1}&=&0\\
 \sum_{I}\xi^I\frac{1}{Z^{I}}\cdot Z_0^{e_1-1}Z_1^{e_2}&=&0\\
\end{array}
\right.\\
&&\\
&\Leftrightarrow  &
\left\{
\begin{array}{lcl}
 \xi^I=0 & & \forall I=(i_0,\dots, i_4) \ / \  i_0> e-1 \\
 \xi^I=0 & & \forall I=(i_0,\dots, i_4) \ / \  i_0> e_1-1 \ \rm{and}  \ i_1 > e_2. \\
\end{array}
\right.
\end{array}
$$
Note that to do this, we suppose $\beta_1\neq a_0\alpha_1.$ Similarly, we find

$$
\begin{array}{ccc}
\xi\in \ker\left(\cdot\frac{\partial F_{\alpha}}{\partial Z_1}\right)\cap\ker\left(\cdot\frac{\partial G_{\beta}}{\partial Z_1}\right)& \Leftrightarrow  &
\left\{
\begin{array}{lcl}
 \xi^I=0 & & \forall I=(i_0,\dots, i_4) \ / \  i_1> e-1 \\
 \xi^I=0 & & \forall I=(i_0,\dots, i_4) \ / \  i_0> e_1 \ \rm{and}  \ i_1> e_2-1 \\
\end{array}
\right.
\end{array}
$$
$$
\begin{array}{ccc}
\xi\in \ker\left(\cdot\frac{\partial F_{\alpha}}{\partial Z_2}\right)\cap\ker\left(\cdot\frac{\partial G_{\beta}}{\partial Z_2}\right)& \Leftrightarrow  &
\left\{
\begin{array}{lcl}
 \xi^I=0 & & \forall I=(i_0,\dots, i_4) \ / \  i_2> e-1 \\
 \xi^I=0 & & \forall I=(i_0,\dots, i_4) \ / \  i_2> e_1-1 \ \rm{and}  \ i_3> e_2  \\
\end{array}
\right.
\end{array}
$$
$$
\begin{array}{ccc}
\xi\in \ker\left(\cdot\frac{\partial F_{\alpha}}{\partial Z_3}\right)\cap\ker\left(\cdot\frac{\partial G_{\beta}}{\partial Z_3}\right)& \Leftrightarrow  &
\left\{
\begin{array}{lcl}
 \xi^I=0 & & \forall I=(i_0,\dots, i_4) \ / \  i_3> e-1 \\
 \xi^I=0 & & \forall I=(i_0,\dots, i_4) \ / \  i_2> e_1 \ \rm{and}  \ i_3> e_2-1 \\
\end{array}
\right.
\end{array}
$$
$$
{\rm and}\ \ \ \begin{array}{ccc}
\xi \in \ker\left(\cdot\frac{\partial F_{\alpha}}{\partial Z_4}\right)\cap\ker\left(\cdot\frac{\partial G_{\beta}}{\partial Z_4}\right)& \Leftrightarrow  &
\left\{
\begin{array}{ccc}
 \xi^I=0 & & \forall I=(i_0,\cdots, i_4) \ / \  i_4>e-1.\\
\end{array}
\right.
\end{array}
$$
Now we just have to observe that any multi-index $I=(i_0,\dots,i_4)$ with $|I|=4a$ satisfies one of those conditions, because otherwise we would get 
$$4e=i_0+\dots+i_4\leqslant 2(e_2+e-2)+e-1 \leqslant 4e-4,$$
which is a contradiction. And therefore, if $\xi\in \im \tva,$ then $\xi^I=0$ for all $I$. Hence, $\im\tva=\{0\}.$
\end{proof}

Observe that from this simple example, one can easily generate other families for which the dimension of the space of holomorphic differential forms jumps.
\begin{corollary}\label{CoroFamily}
For any $n\geqslant 2$, for any $m\geqslant 2$, there is a family of varieties $\mathscr{Y}\to B$ over a curve, of relative dimension $n$, and a point $0\in B$ such that for generic $t\in B,$ $$h^0(Y_0,S^m\Om_{Y_0})> h^0(Y_t,S^m\Om_{Y_t}).$$
\end{corollary}
\begin{proof} Let  $\ccX$ be the family of surfaces constructed in Proposition \ref{PropFamille}. Then it suffices to consider the family $\ccX\times A$ for any $(n-2)-$dimensional abelian variety $A$.
\end{proof}


\section{Varieties with ample  cotangent bundle}\label{SectionAmplitude}
\subsection{Statements}
In this section we prove our main application of the results of Section  \ref{MainSection}. Our statement  is the following partial result towards Debarre's conjecture (denoted by Theorem \ref{theoD} in the introduction).

\begin{theorem}\label{Amplitude}
Let $N,c,e\in\bN$ such that $N\geqslant 2$, $c\geqslant \frac{3N-2}{4}$ and $e\geqslant 2N+3$. If $X\subseteq \bP^N$ be a general complete intersection of codimension $c$ and multidegree $(e,\dots, e)$, then $\Om_X$ is ample.
\end{theorem}
\begin{remark}Observe that the condition $c\geqslant \frac{3N-2}{4}$ can be rephrased by $\codim_{\bP^N}(X)\geqslant 3\dim(X)-2$. \end{remark}
Since ampleness is an open condition (see for instance Theorem 1.2.17 in \cite{Laz04I} and Proposition 6.1.9 in \cite{Laz04II}), it suffices to construct one example of a smooth complete intersection variety  with ample cotangent bundle satisfying the hypothesis of the theorem to prove that the result holds for a general complete intersection variety. We will construct such an example by considering deformations of Fermat type complete intersection varieties. To do so, let us introduce some notation. 

Fix $N\geqslant 2$, $\ep \in \bN$ and $e\in \bN^*$. Set $\bA_\ep:=\bC[Z_0,\dots,Z_N]_\ep.$ For any $s=(s_0,\dots,s_N)\in \bA_\ep^{\oplus N+1}$ we set:
\begin{eqnarray}
F_s(Z)=F(s,Z):=\sum_{i=0}^Ns_iZ_i^e.
\end{eqnarray}
This is a homogeneous equation of degree $e_0:=\ep+e$ and for a general choice of $s,$ it defines a smooth hypersurface in $\bP^N$ which we will denote by $X_s$.
For any $m,a\in \bN$, and any smooth variety $X\subseteq \bP^N$, we will write $\bL_X:=\cO_{\bP(\Om_X)}(1)$ and $\bL_X^m(-a):=\bL_X^{\otimes m}\otimes \pi_X^*\cO_X(-a).$ With those notation we have the following.
\begin{theorem}\label{AmplitudeExemple} Let $N,c,e,\ep,a\in\bN$ such that $N\geqslant 2$, $c\geqslant \frac{3N-2}{4}$ and set $n:=N-c$. Suppose that  $\ep\geqslant 1$  and that $e\geqslant N+1+a+N\ep$. Then for any $0\leqslant j\leqslant N$ there exists $s^j\in \bA_\ep^{\oplus N+1}$ such that $X:=X_{s^1}\cap \cdots \cap X_{s^c}$ is a smooth complete intersection variety such that $\bL_X^n(-a)$ is nef.
\end{theorem}
Let us first prove that this result implies Theorem \ref{Amplitude}.
\begin{proof}[Proof of Theorem \ref{AmplitudeExemple} $\Rightarrow$ Theorem \ref{Amplitude}] Take $a=\ep=1$, $e\geqslant N(1+\ep)+a+1=2N+2$ (this is equivalent to $e_0\gs 2N+3$) and $c\geqslant \frac{3N-2}{4}$. Then Theorem \ref{AmplitudeExemple} implies that there exists a smooth complete intersection variety $X\subset\bP^N$ of codimension $c$ and multidegree $(e,\dots,e)$ such that $\bL^n_X(-1)$ is nef. But this implies that $\bL^n_X$ is ample and therefore that $\bL_X$ is ample, and thus, $\Om_X$ is ample. But since ampleness is an open condition in families, we deduce that a general complete intersection variety in $\bP^N$ of codimension $c$ and multidegree $(e,\dots,e)$ has ample cotangent bundle.
\end{proof}
The proof of Theorem \ref{AmplitudeExemple} is an induction based on the following technical lemma.
\begin{lemma}\label{EliminationLemma}  Let $N,c,e,\ep,a\in\bN$ such that $N\geqslant 2$, $c\geqslant \frac{3N-2}{4}$, $\ep\geqslant 1$ and $e\geqslant N+1+a+N\ep$. Set $n:=N-c$. For any $0\leqslant j\leqslant N$ take $s^j\in \bA_\ep^{\oplus N+1}$ such that $X:=X_{s^1}\cap \cdots \cap X_{s^c}$ is a smooth complete intersection variety. 
 For any $i\in \{0,\dots,N\}$, set $H_i:=(Z_i=0)$ and $W_i:=X\cap H_i$. We look at $\bP(\Om_{W_i})$ as a subvariety of $\bP(\Om_X)$. Then, for a general choice of  $s^1,\dots,s^c$, there exists $E\subseteq \bP(\Om_X)$, such that $\dim E=0$ and such that   $$\bB(\bL_X^n(-a))\subseteq\bigcup_{i=0}^N \bP(\Om_{W_i})\cup E$$
\end{lemma}
The proof of Lemma \ref{EliminationLemma} is the content of Section \ref{ExpliciteDescription} and Section \ref{EliminationSection}. But for now let us explain how to obtain  Theorem \ref{AmplitudeExemple} from Lemma \ref{EliminationLemma}.
\begin{proof}[Proof of Theorem \ref{AmplitudeExemple}] Let $H$ be a hyperplan section of $X$. Fixing $c\geqslant \frac{3N-2}{4}$, the proof is an induction on $n=\dim X$ (or equivalently on $N$). If $n=1$, then $X$ is a curve. By the adjunction formula, $K_X=\cO_{X}(c(e+\ep)-N-1)$. Therefore $\Om_X(-a)=K_X(-a)$ is nef. \\
Now suppose that $n\geqslant 2$. We have to prove that for any irreducible curve $C\subseteq \bP(\Om_X)$, $C \cdot \bL_X^n(-a)\geqslant 0$. Certainly, if $C\not\subset \bB(\bL_X^n(-a))$, $C\cdot \bL_X^n(-a)\geqslant 0$. Now suppose that $C\subseteq \bB(\bL_X^n(-a))$, from Lemma \ref{EliminationLemma} we know that $\bB(\bL_X^n(-a))\subseteq \bigcup_{i=0}^N \bP(\Om_{W_i})\cup E$ for some zero-dimensional set $E$. Therefore,  there exists $i\in \{0,\dots, N\}$ such that $C\subseteq \bP(\Om_{W_i})$. But on can view $W_i$ as a codimension $c$ complete intersection variety in $H_i\cong \bP^{N-1}$ defined by equations of the same type than $X$ (with one less variable). Observe moreover that $\bL_X(-a)|_{\bP(W_i)}= \bL_{W_i}(-a)$ and that if $c\geqslant \frac{3N-2}{4}$, one has $c\geqslant\frac{3(N-1)-2}{4}$. From our  induction hypothesis we therefore  know that $\bL_{W_i}^{n-1}(-a)$ is nef.
Thus, $\bL_{W_i}^{n-1}(-a)\cdot C\gs 0$ and in particular $(n-1)\bL_{W_i}\cdot C\gs a\pi_X^*H\cdot C.$ 
Hence,
$$n\bL_X\cdot C=\frac{n}{n-1}(n-1)\bL_{W_i}\cdot C\gs \frac{n}{n-1}a\pi_X^*H\cdot C\gs a\pi^*H\cdot C.$$
And finally $\bL_X^n(-a)\cdot C\gs 0.$ 
\end{proof}

\subsection{Constructing symmetric differential forms}\label{ExpliciteDescription} To prove Lemma \ref{EliminationLemma} we first need to construct sufficiently many symmetric differential forms   on complete intersection varieties of the above type, this is the purpose of this section. The setting is the following. We fix $N,c,\ep,e,a,r,n\in \bN$ such that $N\geqslant 2$, $N\geqslant c\geqslant\frac{N}{2}$, $a\geqslant 0$, $\ep\geqslant 0$, $e\geqslant a+N(\ep+1)+1$, $r=e-1$ and $n=N-c.$ For any $1\leqslant j\leqslant c$, we take $s^{j}\in \bA_\ep^{N+1}$ such that for any $1\leqslant p\leqslant c$ and for any $I:=(i_1,\dots,i_p)\in \{1,\dots,c\}_{\neq}^p$, the set $X^I:=X_{s^{i_1}}\cap\cdots\cap X_{s^{i_p}}$ is a smooth complete intersection variety. Of course, this last condition holds if the $s^{i_j}$'s are general. We also denote $X:=X^{(1,\dots,c)}$. For any $0\leqslant i\leqslant N$, any  $1\leqslant j\leqslant c$ and any $v\in \bA_\ep$, we set $a_i(v):=Z_iv$ and $\al_i(v):=Z_idv+evdZ_i.$ So that :
\begin{eqnarray}F({s^{j}},Z)=\sum_{i=0}^N a_i(s_i^{j})Z_i^r\ \ \ \text{and}\ \ \ dF(s^j,Z):=dF_{s^{j}}(Z)=\sum_{i=0}^N \al_i(s_i^{j})Z_i^r\label{EquationFermat}\end{eqnarray}
For any $i\in \{0,\dots,n\}$ set $U_i:=(Z_i\neq 0)$, $\mathfrak{U}=(U_i)_{0\leqslant i\leqslant N}$ and $\mathfrak{U}_X:=(U_i\cap X)_{0\leqslant i\leqslant N}$, $\mathfrak{U}_X$ is an open covering of $X$. Our first result is the following.
\begin{lemma}\label{TildeDiffForms}
With the above notations, for any $I=(i_1,\dots,i_n)\in \{1,\dots,c\}^n_{\neq}$, and for any degree $e-a-N\ep-N-1$ homogeneous polynomial $P\in\bC[Z_0,\dots,Z_N],$ there is a non-zero element $\tom^{I,P}\in H^0(X,S^n\tOm_X(-a))$ such that, when computed in \v{C}ech cohomology one has  $\tom^{I,P}=(\tom_0^{I,P},\dots, \tom_N^{I,P})\in \check{H}^0(\fU_X,S^n\tOm_X(-a))$ with 
$$\tom_j^{I,P}=\frac{(-1)^j P}{Z_i^r}\left|\begin{array}{cccccc}a_0(s_0^1)&\cdots&a_{j-1}(s_{j-1}^1)&a_{j+1}(s_{j+1}^1)&\cdots&a_N(s_N^1)\\ \vdots&&\vdots&\vdots&&\vdots\\ a_0(s_0^c)&\cdots&a_{j-1}(s_{j-1}^c)&a_{j+1}(s_{j+1}^{c})&\cdots&a_N(s_N^{c})\\ \al_0(s_0^{i_1})&\cdots&\al_{j-1}(s_{j-1}^{i_1})&\al_{j+1}(s_{j+1}^{i_1})&\cdots&\al_N(s_N^{i_1})\\ \vdots&&\vdots&\vdots&&\vdots\\ \al_0(s_0^{i_n})&\cdots&\al_{j-1}(s_{j-1}^{i_n})&\al_{j+1}(s_{j+1}^{i_n})&\cdots&\al_N(s_N^{i_n})\end{array}\right|.$$
\end{lemma}
\begin{proof} For simplicity, for any $i\in \{1,\dots,N\}$ and any $j\in \{1,\dots,c\}$, we write $a_i^j:=a_i(s_i^j)$, $\al_i^j:=\al_i(s_i^j)$ and $F_j=F_{s^j}.$ As in the statement, take $I=(i_1,\dots,i_n)\in \{1,\dots,c\}^n_{\neq}$,  and take $(i_{n+1},\dots,i_c)\in \{1,\dots,c\}^{c-n}_{\neq}$ such that $\{i_1,\dots,i_c\}=\{1,\dots,c\}.$  For any $j\in \{0,\dots, c\}$ write $I_j=(i_1,\dots, i_j).$
We can apply Theorem  \ref{LimiteSetting} and Theorem \ref{LimitePair} to the simple $\lambda$-setting $\Si=(X,\lambda^0,\dots,\lambda^c)$ where $\lambda^i=\varnothing$ if $i\neq n$ and $\lambda^n=n$ (note that $q(\Si)=0$). Here the order in which we intersect the different hypersurfaces to obtain $X$ is crucial. We obtain an injection 
$$H^0\left(X,S^n\tOm_{X^{I_n}}|_X\otimes \cO_{X}(-a)\right)\stackrel{\tilde{\varphi}}{\rightarrow}H^N\left(\bP^N,\cO_{\bP^N}(-a-Ne_0)\right)$$
such that \begin{eqnarray}\im(\tphi)=\left(\bigcap_{j=n+1}^c\ker(\cdot F_{{i_j}})\right)\bigcap\left(\bigcap_{i=1}^n\ker(\cdot F_{i_j})\cap\ker(\cdot dF_{i_j})\right).\label{description}\end{eqnarray}
Observe that our hypothesis on the degree ensures that $e-a-N\ep-N-1\geqslant 0$. Fix any degree $e-a-N\ep-N-1$ polynomial $P$. Then we get a well defined element $$\tom^{I,P}_{0,\dots, N}:=\frac{P}{Z_0^r\cdots Z_N^r}\in \check{H}^N\left(\fU,\cO_{\bP^N}(-a-Ne_0)\right).$$ 
From \eqref{EquationFermat} and \eqref{description} we see that $\tom^{I,P}_{0,\dots, N}\in \im(\tphi)$. To obtain the explicit description of $\tphi^{-1}(\tom^{I,P}_{0\cdots N})\in  \check{H}^0\left(\fU_X,S^n\tOm_{X^{I_n}}|_X(-a)\right)$ we have to describe explicitly the inclusion $\tphi$, to do so, we have to unravel the proof of Theorem \ref{LimiteSetting} in our situation. This inclusion is described in Proposition \ref{ChainInclusion}, applying this proposition in our situation, we see that $\tphi$ is obtained as the composition of the following chain of inclusions:
\begin{eqnarray*}
H^0(X,S^n\tOm_{X^{I_n}}|_X(-a))&\hookrightarrow& H^1\left(X,S^{n-1}\tOm_{X^{I_{n-1}}}|_{X}(-a-e_0)\right)\\
&\hookrightarrow & H^2\left(X,S^{n-2}\tOm_{X^{I_{n-2}}}|_{X}(-a-2e_0)\right)\\
&\vdots&\ \ \ \ \ \ \ \ \ \ \vdots\\
&\hookrightarrow & H^{n-1}(X,\tOm_{X^{I_1}}|_{X}(-a-(n-1)e_0))\\
&\hookrightarrow & H^{n}(X,\cO_X(-a-ne_0))\\
&\hookrightarrow & H^{n+1}(X^{I_{c-1}},\cO_{X^{I_{c-1}}}(-a-(n+1)e_0))\\
&\vdots&\ \ \ \ \ \ \ \ \ \ \vdots\\
&\hookrightarrow & H^{N}(\bP^N,\cO_{\bP^N}(-a-Ne_0))\\
\end{eqnarray*}
For any $0\leqslant \ell\leqslant N$, we denote by $\tom^{\ell}$ element in the $H^\ell$ group in the above chain of inclusion whose image in $H^N(\bP^N,\cO_{\bP^N}(-a-N(e+\ep)))$ under the above inclusions is $\tom^{I,P}.$ Moreover, for any $0\leqslant\ell\leqslant N$ and for any $K=(k_0,\dots,k_\ell)\in \{0,\dots,N\}_{<}^{\ell+1}$ if we let $ K^\circ=(k_{\ell+1},\dots,k_N)\in \{0,\dots,N\}^{N-\ell}_<$ such that $\{k_0,\dots,k_N\}=\{0,\dots,N\}$, we set $\ep(K)$ to be the signature of the following permutation 
$$\si_K:=\left(\begin{array}{cccc}0&1&\cdots & N\\ k_0&k_1&\cdots&k_N\end{array}\right),$$ 
and for any $m\in \{1,\dots,c\}$, we  set $L^{m}_K:=(a^{m}_{k_{\ell+1}},\dots,a^{m}_{k_{N}})$ and $\La^{m}_K:=(\al^{m}_{k_{\ell+1}},\dots,\al^{m}_{k_{N}}).$ We also set $\bX_K:=\frac{1}{Z_{k_0}^r\cdots Z_{k_\ell}^r}$.\\ 

Let $\ell\in \{0,\dots,N\},$ we are going to prove by induction that $\tom^{\ell}$ can be represented in \v{C}ech cohomology by the following cocycle $(\tom^\ell_K)_{K\in \{0,\dots, N\}^{\ell+1}_{<}}:$
\begin{enumerate}
\item If $n\leqslant \ell\leqslant N$, write $m:=N-\ell$, then $\tom^\ell_K=(-1)^{N(N-\ell+1)}\ep(K)P\left|\begin{array}{c}L_K^{i_1}\\ \vdots\\ L_K^{i_m}\end{array}\right|\bX_K.$
\item If $0\leqslant \ell\leqslant n$, write $m:=n-\ell$, then $\tom^\ell_K=(-1)^{N(N-\ell+1)}\ep(K)P\left|\begin{array}{c}L_K^{i_1}\\ \vdots\\ L_K^{i_c}\\ \La_K^{i_1}\\ \vdots \\ \La_K^{i_{m}}\end{array}\right|\bX_K.$
\end{enumerate}
Of course, if we look at $\ell=0$ in the last case, up to a sign, we get the description announced in the statement of the lemma, and it suffices to look at $\tom^{I,P}$ as an element in $H^0(X,S^n\tOm_X(-a))$ by considering its image under the natural restriction map $H^0(X,S^n\tOm_{X^{I_n}}|_X(-a))\to H^0(X,S^n\tOm_X(-a)).$ \\

We start by computing a \v{C}ech cocycle for $\tom^{N-1}$. Recall that one has the twisted restriction short exact sequence
$$0\to \cO_{\bP^N}(-a-Ne_0)\stackrel{\cdot F_{i_1}}{\rightarrow} \cO_{\bP^N}(-a-(N-1)e_0)\to \cO_{X_{i_1}}(-a-(N-1)e_0)\to 0, $$
which yields the following in cohomology:
$$0\to H^{N-1}(X_{i_1},\cO_{X_{i_1}}(-a-(N-1)e_0))\stackrel{\delta_{\cdot F_{i_1}}}{\rightarrow} H^N(\bP^N, \cO_{\bP^N}(-a-Ne_0))\stackrel{\cdot F_{i_1}}{\rightarrow} H^N(\bP^N,\cO_{\bP^N}(-a-(N-1)e_0)). $$
By the very definition of $\tom^{N-1}$, we have $\delta_{\cdot F_{i_1}}(\tom^{N-1})=\tom^N=\tom^{I,P}$.
As explained in Section \ref{Cech}, to compute a cocycle for $\tom^{N-1}$, we compute $\tom^{N}_{0,\dots, N}\cdot F_{i_1}$ and write it as the \v{C}ech differential of the desired cocycle. We have 
\begin{eqnarray*}
\tom^N_{0,\dots, N}\cdot F_{i_1}&=&P\cdot \bX_{0,\dots,N}\cdot \sum_{j=0}^Na_j^{i_j}Z_j^r\\
&=& P\cdot \sum_{j=0}^Na_j^{i_j}\bX_{0,\dots,\hat{j},\dots,N}= P\cdot \sum_{j=0}^N(-1)^j\left((-1)^j(-1)^{N(N-1+1)}a_j^{i_j}\bX_{0,\dots,\hat{j},\dots,N}\right)
\end{eqnarray*}
If we write $K=(0,\dots,\hat{j},\dots,N)$, we have $\ep(K)=(-1)^j$ and $L^{i_1}_{K}=(a^{i_1}_j)$ and $\bX_K=\bX_{0,\dots,\hat{j},\dots,N}$. Therefore, we see that $\tom^N_{0,\dots, N}\cdot F_{i_1}=\check{d}\left((\tom_K^{N-1})_{K\in\{0,\dots,N\}^{N}_<}\right)$ where as announced $$\tom_K^{N-1}=(-1)^{N(N-1+1)}\ep(K)P|L_K^{i_1}|\bX_K.$$
We will only treat one case of the induction, the other case is treated the exact same way. Let $n< \ell\ls N$, let $m=N-\ell$ and suppose that we know that $\tom^\ell$ can be represented by a cocycle $(\tom_K^\ell)_{K\in \{0,\dots,N\}^{\ell+1}_<}$ as in the case $1$. We have to prove that $\tom^{\ell-1}$ can be represented by a cocycle of the form described in case $1.$ We have $$\tom^\ell\in H^\ell\left(X^{I_{m}},\cO_{X^{I_m}}(-a-\ell e_0)\right),$$ 
and we have the following short exact sequence 
\begin{eqnarray*}
0\to \cO_{X^{I_m}}(-a-\ell e_0)& \stackrel{\cdot F_{i_{m+1}}}{\to}& \cO_{X^{I_m}}(-a-(\ell-1)e_0) \to \cO_{X^{I_{m+1}}}(-a-(\ell-1)e_0)\to 0.
\end{eqnarray*}
In cohomology, it yields an injection (which is precisely the one appearing in the above chain of inclusions)
$$H^{\ell-1}\left(X^{I_{m+1}},\cO_{X^{I_{m+1}}}(-a-(\ell-1)e_0)\right)\stackrel{\delta_{\cdot F_{i_{m+1}}}}{\hookrightarrow} H^{\ell}\left(X^{I_{m}},\cO_{X^{I_{m}}}(-a-\ell e_0)\right)$$
To compute a \v{C}ech cohomology cocycle of $\om^{\ell-1}$, we proceed as before. For any $K\in \{0,\dots, N\}_<^{\ell+1}$ we are going to compute $\om_K^\ell\cdot F_{i_{m+1}}$ modulo $F_{i_1},\dots,F_{i_m}.$
Fix $K=(k_0,\dots,k_\ell)\in\{0,\dots,N\}_<^{\ell+1}$ and as usual, take $K^\circ=(k_{\ell+1},\dots,k_N)\in\{0,\dots,N\}^{N-\ell}_<$ such that $K\cup K^\circ=\{0,\dots,N\}.$ Because we are working modulo $F_{i_1},\dots,F_{i_m},$ we have
$$\left\{\begin{array}{ccc}F_{i_1}&=&0\\ \vdots &&\vdots \\ F_{i_m}&=&0 \\ F_{i_{m+1}}-F_{i_{m+1}}&=&0\end{array}\right.\Leftrightarrow \left\{\begin{array}{ccc}\sum_{j=\ell+1}^Na_{k_j}^{i_1}Z_{k_j}^r&=&-\sum_{j=0}^\ell a_{k_j}^{i_1}Z_{k_j}^r\\ \vdots &&\vdots \\ \sum_{j=\ell+1}^Na_{k_j}^{i_m}Z_{k_j}^r&=&-\sum_{j=0}^\ell a_{k_j}^{i_m}Z_{k_j}^r \\ \sum_{j=\ell+1}^N a_{k_j}^{i_{m+1}}Z_{k_j}^r-F_{i_{m+1}}&=&-\sum_{j=0}^\ell a_{k_j}^{i_{m+1}}Z_{k_j}^r\end{array}\right.$$
We may write this as follows:
$$\left(\begin{array}{cccc}a^{i_1}_{k_{\ell+1}}&\cdots&a^{i_1}_{k_N}&0\\  \vdots & & \vdots & \vdots \\ a^{i_m}_{k_{\ell+1}}&\cdots&a^{i_m}_{k_N}&0\\ a^{i_{m+1}}_{k_{\ell+1}}&\cdots&a^{i_{m+1}}_{k_N}&-1\end{array}\right)\left(\begin{array}{c}Z^r_{k_{\ell+1}}\\ \vdots \\ Z^r_{k_N}\\ F_{i_{m+1}}\end{array}\right)=\left(\begin{array}{c}-\sum_{j=0}^\ell a_{k_j}^{i_1}Z_{k_j}^r\\ \vdots \\ -\sum_{j=0}^\ell a_{k_j}^{i_m}Z_{k_j}^r \\ -\sum_{j=0}^\ell a_{k_j}^{i_{m+1}}Z_{k_j}^r\end{array}\right).$$
By a simple application of Cramer's rule we obtain
$$F_{i_m}\left|\begin{array}{cccc}a^{i_1}_{k_{\ell+1}}&\cdots&a^{i_1}_{k_N}&0\\ \vdots & & \vdots & \vdots \\ a^{i_m}_{k_{\ell+1}}&\cdots&a^{i_m}_{k_N}&0\\ a^{i_{m+1}}_{k_{\ell+1}}&\cdots&a^{i_m}_{k_N}&-1\end{array}\right|=\left|\begin{array}{cccc}a^{i_1}_{k_{\ell+1}}&\cdots&a^{i_1}_{k_N}&-\sum_{j=0}^\ell a_{k_j}^{i_1}Z_{k_j}^r\\ \vdots & & \vdots & \vdots \\ a^{i_m}_{k_{\ell+1}}&\cdots&a^{i_m}_{k_N}&-\sum_{j=0}^\ell a_{k_j}^{i_m}Z_{k_j}^r\\ a^{i_{m+1}}_{k_{\ell+1}}&\cdots&a^{i_{m+1}}_{k_N}&-\sum_{j=0}^\ell a_{k_j}^{i_{m+1}}Z_{k_j}^r\end{array}\right|$$
which yields 
$$\left|\begin{array}{c}L_K^{i_1}\\  \vdots \\ L^{i_m}_K\end{array}\right|F_{i_{m+1}}=\sum_{j=0}^\ell \left|\begin{array}{cccc}a^{i_1}_{k_{\ell+1}}&\cdots&a^{i_1}_{k_N}&a_{k_j}^{i_1}\\  \vdots & & \vdots & \vdots \\ a^{i_{m+1}}_{k_{\ell+1}}&\cdots&a^{i_{m+1}}_{k_N}&a_{k_j}^{i_{m+1}}\end{array}\right|Z_{k_j}^r.$$
We just have to be a bit careful with the signs. For any $j\in \{0,\dots, \ell\}$ we set $K_j=K\setminus \{k_j\}$. We also take $\nu_j\in \{\ell+1,\dots, N-1\}$ such that $k_{\nu_j}<k_j<k_{\nu_j+1}$ (if $k_j< k_{\ell+1} $, we just take $\nu_j=\ell$). With this notation, we have $K^\circ_j=(k_{\ell+1},\dots, k_{\nu_j},k_j,k_{\nu_j+1},\dots, k_N)$ (or $K^\circ_j=(k_j,k_{\nu_j+1},\dots, k_N) $ if $k_j<k_{\ell+1}$). A straightforward computation shows that $\si_{K_j}=(k_j,k_{j+1},\dots, k_{\nu_j})^{-1}\circ \si_K$, hence $\ep(K_j)=(-1)^{\nu_j-j}\ep(K).$ On the other hand, observe that 
$$\left|\begin{array}{cccc}a^{i_1}_{k_{\ell+1}}&\cdots&a^{i_1}_{k_N}&a_{k_j}^{i_1}\\  \vdots & & \vdots & \vdots \\ a^{i_{m+1}}_{k_{\ell+1}}&\cdots&a^{i_{m+1}}_{k_N}&a_{k_j}^{i_{m+1}}\end{array}\right|
=(-1)^{N-\nu_j}\left|\begin{array}{c}L_{K_j}^{i_1}\\  \vdots \\ L_{K_j}^{i_{m+1}} \end{array}\right|.$$
Now we are ready for our computation:
\begin{eqnarray*}
\om^\ell_{K}\cdot F_{i_{m+1}}&=&(-1)^{N(N-\ell)}\ep(K)P\bX_K\left|\begin{array}{c}L_K^{i_1}\\  \vdots \\ L^{i_{m}}_K\end{array}\right|F_{i_{m+1}}= (-1)^{N(N-\ell)}\ep(K)P\bX_K\sum_{j=0}^\ell (-1)^{N-\nu_j}\left|\begin{array}{c}L_{K_j}^{i_1}\\ \vdots \\ L^{i_{m+1}}_{K_j}\end{array}\right|Z_{k_j}^r\\
&=&(-1)^{N(N-\ell)}P\sum_{j=0}^\ell (-1)^{N}\ep(K)(-1)^j(-1)^{j-\nu_j}\left|\begin{array}{c}L_{K_j}^{i_1}\\ \vdots \\ L^{i_{m+1}}_{K_j}\end{array}\right| \bX_{K_j}\\
&=&\sum_{j=0}^\ell(-1)^{N(N-\ell+1)}\ep(K_j)(-1)^jP\left|\begin{array}{c}L_{K_j}^{i_1}\\  \vdots \\ L^{i_{m+1}}_{K_j}\end{array}\right| \bX_{K_j}.
\end{eqnarray*}
Hence $\left(\tom^\ell\cdot F_{i_{m+1}}\right)_K=\check{d}\left(\left(\tom^{\ell-1}_{K'}\right)_{K'\in\{0,\dots,N\}^\ell_<}\right),$ where for any $K'\in\{0,\dots,N\}_<^{\ell}$, $$\om^{\ell-1}_{K'}=(-1)^{N(N-\ell+1)}\ep(K')P\left|\begin{array}{c}L_{K'}^{i_1}\\  \vdots \\ L^{i_{m+1}}_{K'}\end{array}\right| \bX_{K'}$$ as announced.
\end{proof}
The last step is to use the tilde-symmetric differential forms constructed in Lemma \ref{TildeDiffForms} to get usual symmetric differential forms. Recall that in our situation, from the exact sequence 
$$0\to S^n\Om_X(-a)\to S^n\tOm_X(-a)\to S^{n-1}\tOm_X(-a)\to 0$$
and the vanishing $H^0\left(X,S^{n-1}\tOm_X(-a)\right)$ (deduced from Lemma \ref{VanishingLemma}), we get an isomorphism 
$$H^0\left(X,S^n\Om_X(-a)\right)\cong  H^0\left(X,S^n\tOm_X(-a)\right).$$
Therefore we just have to compute the image $\om^{I,P}$ of $\tom^{I,P}$ under this isomorphism. To do this we just have to dehomogenize $\tom^{I,P}.$ For simplicity, we only work in the $U_0=(Z_0\neq 0)$ chart. We use the following notation: for any $i\in \{1,\dots,N\}$, $z_i=\frac{Z_1}{Z_0}$, so that $dZ_i=Z_0dz_i+z_idZ_0.$ Let us make an elementary observation. Take $\ell\gs 1$, and let $G\in \bC[Z_0,\dots,Z_N]$ be a homogeneous degree $\ell$ polynomial. Let $g$ be the dehomogenized of $G$ with respect to $Z_0$. Then we have $$dG=\ell Z_0^{\ell-1}gdZ_0+Z_0^\ell dg.$$ Indeed,
\begin{eqnarray*}
Z_0dG&=&\sum_{i=0}^NZ_0\frac{\partial G}{\partial Z_i}dZ_i = Z_0\frac{\partial G}{\partial Z_0}dZ_0+\sum_{i=1}^N\frac{\partial G}{\partial Z_i}(Z_0^2dz_i+Z_idZ_0)\\
&=& \sum_{i=0}^N\frac{\partial G}{\partial Z_i}Z_idZ_0+Z_0^2\sum_{i=1}^N\frac{\partial G}{\partial Z_i}dz_i=\ell GdZ_0+Z_0^2\sum_{i=1}^NZ_0^{\ell-1}\frac{\partial g}{\partial z_i}dz_i= \ell Z_0^{\ell}gdZ_0+Z_0^{\ell+1} dg.
\end{eqnarray*}
We introduce some more notation. For any $u\in \bA_\ep\cong\bC[z_1,\dots,z_N]_{\ls\ep}$ and any $q\in \{1,\dots,N\}$  we set 
 \begin{eqnarray*}
 b_q(u):=z_qu \ \ \ \text{and} \ \ \ \be_q(u):=z_qdu+eudz_q.
 \end{eqnarray*}
For any $1\ls j\ls c$ and for any $0\ls i\ls N$ we let $t_i^{j}$ (resp. $f_{j}$) to be the dehomogenized of $s_{i}^{j}$ (resp. $F_j$) with respect to $Z_0$.
Therefore we have $a_{i}(s_i^j)=Z_is_i^{j}=Z_0^{\ep+1}z_it_i^{j}=Z_0^{\ep+1}b_i(t_i^j)$ and 
\begin{eqnarray*}
\al_i(s_i^j)&=& Z_ids_i^{j}+es_i^{j}dZ_i=Z_0z_i(\ep Z_0^{\ep-1}t_i^{j}dZ_0+Z_0^\ep dt_i^{j})+eZ_0^\ep t_i^{j}(z_idZ_0+Z_0dz_i)\\
&=& e_0Z_0^\ep z_it_i^{j}dZ_0+Z_0^{\ep+1}(z_idt_i^{j}+et_i^{j}dz_i)=e_0Z_0^\ep b_i(t_i^j)dZ_0+Z_0^{\ep+1}\be_i(t_i^j).
\end{eqnarray*}
Therefore, by the elementary properties of the determinant we get 
$$\tom^{I,P}_0=\frac{P}{Z_0^r}\left|\begin{array}{ccc}a_1(s_{1}^{1})&\cdots&a_N(s_N^{1})\\ \vdots&&\vdots\\ a_1(s_{1}^{c})&\cdots&a_N(s_N^{c})\\ \al_1(s_{1}^{i_1})&\cdots&\al_N(s_N^{i_1})\\ \vdots&&\vdots\\ \al_1(s_{1}^{i_n})&\cdots&\al_N(s_N^{i_n})\end{array}\right|=PZ_0^{N(\ep+1)-r}\left|\begin{array}{ccc}b_1(t_{1}^{1})&\cdots&b_N(t_N^{1})\\ \vdots&&\vdots\\ b_1(t_{1}^{c})&\cdots&b_N(t_N^{c})\\ \be_1(t_{1}^{i_1})&\cdots&\be_N(t_N^{i_1})\\ \vdots&&\vdots\\ \be_1(t_{1}^{i_n})&\cdots&\be_N(t_N^{i_n})\end{array}\right|.$$
This completes the proof of the following.
\begin{proposition}\label{DiffForms} With the above notation. Take a homogenous polynomial $P\in \bC[Z_0,\dots,Z_N]$ of degree $e-a-N\ep-N-1$ and let $Q\in \bC[z_1,\dots,z_N]$ be the dehomogenization of $P$ with respect to $Z_0.$ Then, for any $I=(i_1,\dots,i_n)\in\{1,\dots,c\}_{\neq}^{n}$, the symmetric differential form
$$\om_0^{I,Q}(t):=Q\left|\begin{array}{ccc}b_1(t_{1}^{1})&\cdots&b_N(t_N^{1})\\ \vdots&&\vdots\\ b_1(t_{1}^{c})&\cdots&b_N(t_N^{c})\\ \be_1(t_{1}^{i_1})&\cdots&\be_N(t_N^{i_1})\\ \vdots&&\vdots\\ \be_1(t_{1}^{i_n})&\cdots&\be_N(t_N^{i_n})\end{array}\right|\in H^0(U_0\cap X,S^n\Om_{X})$$
extends to a twisted symmetric differential form $\om^{I,P}\in H^0(X,S^n\Om_X(-a))$. 
\end{proposition}
\subsection{Estimating the base locus}\label{EliminationSection}
\subsubsection{Dimension count lemmas}
We will need some elementary results. The following is classical, we provide a proof because we will use the idea of it again afterwards. 
\begin{lemma} Let $p,q\in \bN$ such that $n\ls p$. Fix a rank $n$ matrix $A\in \Mat_{n,p}(\bC).$ For any $r\in \bN^*$ with $r<n$, let 
$$\Si^{p,q}_r:=\left\{B\in \Mat_{p,q}(\bC)\ \text{such that}\ \ \rk{AB}\ls r\right\}.$$
Then, $\dim \Si_r^{p,q}\ls pq-(q-r)(n-r).$
\end{lemma}
\begin{proof} Set $$\De:=\left\{(B,\Ga)\in \Mat_{p,q}(\bC)\times \Grass(r,\bC^n)\ \text{such that}\ \im(AB)\subseteq \Ga\right\}$$
Let $\pr_1:\De\to \Mat_{p,q}(\bC)$ and $\pr_2:\De\to \Grass(r,\bC^n) $ be the natural projections. For any $\Ga\in \Grass(r,\bC^n)$ we consider $\De_\Ga=\pr_1(\pr_2^{-1}(\Ga)).$ Of course, 
\begin{eqnarray*}
\De_\Ga=\left\{B\in \Mat_{p,q}(\bC)\ \text{such that}\ \im(AB)\subseteq \Ga\right\}=\left\{B\in \Mat_{p,q}(\bC)\ \text{such that}\ \im(B)\subseteq A^{-1}(\Ga)\right\}\end{eqnarray*}
But since $\rk A=n,$ $A$ is surjective and therefore $\dim(A^{-1}(\Ga))=r+p-n$. Therefore, $\De_\Ga\cong \Hom(\bC^q,\bC^{r+p-n})$ and thus  $\dim (\De_\Ga)=q(r+p-n).$ In particular, $$\dim \De= q(r+p-n)+\dim\Grass(r,\bC^n)=q(r+p-n)+r(n-r)=qp-(q-r)(n-r).$$
Since $\Si_r^{p,q}=\pr_1(\De),$ the result follows.
\end{proof}
From the previous lemma one can easily deduce the following.
\begin{corollary}\label{determinental} Let $M,N,c\in \bN^*$ such that $c\ls N$. Let $\ell_1,\dots, \ell_N\in (\bC^M)^\vee$ be non-zero linear forms.  For any $r\in \bN^*$ such that $r<c,$ let 
$$\Si_r=\left\{(x_i^j)^{1\leqslant j\ls c}_{1\ls i \ls N}\in (\bC^M)^{Nc} \ \text{such that } \rk\left(\begin{array}{ccc}\ell_1(x_1^1)&\cdots & \ell_N(x_N^1)\\ \vdots & &\vdots \\ \ell_1(x_1^c)&\cdots & \ell_N(x_N^c)\end{array}\right)\ls r\right\}.$$ 
Then, $\dim(\Si_r)\ls MNc-(N-r)(c-r).$ 
\end{corollary}
The following lemma will be crucial to us.
\begin{lemma}\label{FormesIndep}
Take $c,n,M\in \bN^*$ such that $n\ls c$ and let $N:=n+c$. Take linear forms $\ell_1,\dots, \ell_n,\la_1,\dots ,\la_n\in (\bC^M)^\vee$ such that for any $i\in \{1,\dots,c\}$, $\la_i\neq 0$ and $\ker\ell_i\neq \ker \la_i$. Fix $A\in \Gl_c(\bC)$ and  $B\in \Mat_{c}(\bC)$. For any $x=(x_i^j)^{1\ls j\ls c}_{1\ls i\ls n}\in (\bC^M)^{cn}$ consider 
$$S\left(x\right)=\left(\begin{array}{ccc|ccc}&&&\ell_1(x_1^1)&\cdots&\ell_n(x_n^1)\\&A&&\vdots&&\vdots \\ &&&\ell_1(x_1^c)&\cdots&\ell_n(x_n^c)\\ \hline &&&\la_1(x_1^1)&\cdots&\la_n(x_n^1)\\&B&&\vdots&&\vdots \\ &&&\la_1(x_1^c)&\cdots&\la_n(x_n^c) \end{array}\right).$$
For any $r\in \{c,\dots, N-1\}$, let $$\Si_r:=\left\{x\in (\bC^M)^{cn}\ \text{such that }\ \rk S\left(x\right)\ls r\right\}\subseteq (\bC^M)^{cn}.$$
Then, $\dim(\Si_r)\leqslant Mnc-(N-r)(2c-r).$ 
\end{lemma}
\begin{proof}
Using elementary operations on the $c$ first  columns, we see that we may suppose that $A=I_c$ is the identity matrix of size $c$. We will write everything in matricial notation. For any $i\in \{1,\dots, M\}$ any $j\in \{1,\dots, c\}$ and any $x=(x_i^j)^{1\ls j\ls c}_{1\ls i\ls n}\in (\bC^M)^{cn}$ we let $$L_i=(\ell_{i,1},\dots, \ell_{i,M})\in (\bC^M)^\vee,\ \ \La_i=(\la_{i,1},\dots, \la_{i,M})\in (\bC^M)^\vee \ \text{and} \ X_{i}^j=\left(\begin{array}{c}x_{i,1}^j\\ \vdots\\ x_{i,M}^j\end{array}\right)\in \bC^M$$
such that $\ell_i(x_i^j)=L_iX_i^j$ and $\la_i(x_i^j)=\La_i X_i^j.$ For any matrix $Q$ we denote the transposed of $Q$ by ${}^tQ.$ Set moreover
$$L:=\left(\begin{array}{ccc}{}^tL_1&\cdots &0\\ \vdots &\ddots&\vdots \\ 0&\cdots &{}^tL_n\end{array}\right),\ \ \La:=\left(\begin{array}{ccc}{}^t\La_1&\cdots &0\\ \vdots &\ddots&\vdots \\ 0&\cdots &{}^t\La_n\end{array}\right)\ \text{and}\ X:=\left(\begin{array}{ccc}{}^tX_1^1&\cdots &{}^tX_n^1\\ \vdots &&\vdots \\ {}^tX_1^c&\cdots &{}^tX_n^c\end{array}\right).$$
 With those notations, $S(x)=\left(\begin{array}{cc}I_c&XL\\ B& X\La\end{array}\right),$ and by elementary operations on the lines we get that 
 $$\rk(S(x))=\rk\left(\begin{array}{cc}I_c&XL\\ 0& X\La-BXL\end{array}\right)=c+\rk(X\La-BXL).$$
 Write $B=(b_{i,j})_{1\ls i,j\ls c}.$ By a straightforward computation we obtain that for any $j\in \{1,\dots, n\}$ the $j$-th column of the matrix $X\La-BXL$ is exactly $K_jX_j$ where $$X_j=\left(\begin{array}{c}X_j^1\\ \vdots\\ X_j^c\end{array}\right)\ \ \ \text{and} \ \ \
 K_j=\left(\begin{array}{ccc}\La_j&\cdots&0\\ \vdots&\ddots&\vdots\\ 0&\cdots & \La_j\end{array}\right)-\left(\begin{array}{ccc}b_{11}L_j&\cdots&b_{1c}L_j\\ \vdots&&\vdots\\ b_{c1}L_j&\cdots & b_{cc}L_j\end{array}\right)$$
In particular, we have $X\La-BXL=\left(\begin{array}{ccc}K_1X_1&\cdots&K_nX_n\end{array}\right).$ An easy computation shows that under our hypothesis on the $\ell_i'$s and the $\la_i'$s, $\rk K_j=c$ for all $j\in \{1,\dots,n\}.$ The end of the proof goes as in Lemma~\ref{determinental}. Let
$$\De:=\left\{(X,\Ga)\in (\bC^M)^{nc}\times \Grass(r-c,\bC^c)\ \text{such that} \ \im(X\La-BXL)\subseteq \Ga\right\}$$
 Let $\pr_1:\De\to (\bC^M)^{cn}$ and $\pr_2:\De\to \Grass(r-c,\bC^c) $ be the natural projections. For any $\Ga\in \Grass(r-c,\bC^c)$ let $\De_\Ga=\pr_1(\pr_2^{-1}(\Ga)).$ Then 
 \begin{eqnarray*}
 \De_{\Ga}&=&\left\{X\in (\bC^M)^{nc}\ \text{such that} \ \im(X\La-BXL)\subseteq \Ga\right\}\\
 &=&\left\{X\in (\bC^M)^{nc}\ \text{such that} \ \im\left(\begin{array}{ccc}K_1X_1&\cdots&K_nX_n\end{array}\right)\subseteq \Ga\right\}\\
 &=&\left\{X\in (\bC^M)^{nc}\ \text{such that} \ \im(X_1)\subseteq K_1^{-1}(\Ga),\dots, \im(X_n)\subseteq K_n^{-1}(\Ga)\right\}
 \end{eqnarray*}
 Since for each $j\in \{1,\dots,n\},$ $K_j$ is surjective, we obtain that $\dim(K_j^{-1}(\Ga))=r-c+Mc-c=Mc+r-2c.$
 In particular
$ \De_\Ga\cong \Hom(\bC,\bC^{Mc+r-2c})^{\oplus n},$
which implies that $\dim(\De_\Ga)=n(Mc+r-2c)$. And therefore,
$$\dim(\De)=n(Mc+r-2c)+\dim\Grass(r-c,\bC^c)=n(Mc+r-2c)+(r-c)(2c-r)=Mnc-(n+c-r)(2c-r).$$
But since $\Si_r=\pr_1(\De)$, this concludes the proof of the lemma.
 \end{proof}
We are also going to need the following elementary algebraic geometry lemma.
\begin{lemma}\label{stratification} Let $X$ and $Y$ be two algebraic varieties (not necessarily irreducible). Let $f:X\to Y$ be a regular map. Suppose that $Y=Y_0\cup \cdots\cup Y_r$ such that $Y_i\cap Y_j=\varnothing$  if $i\neq j$. For each $i\in \{0,\dots,r\}$ take $n_i\in \bN$. Suppose that for each $y\in Y_i$ and for each $i\in \{0,\dots,r\}$ , $\dim X_y\leqslant n_i$. Then 
$$\dim X\leqslant \max_{0\ls i\ls r}(\dim(Y_i)+n_i).$$ 
\end{lemma}
\subsubsection{Proof of Lemma \ref{EliminationLemma}}
We are now in position to prove Lemma \ref{EliminationLemma}. It will be a immediate consequence of the following more precise statement.
\begin{lemma}\label{RefinedElimitationLemma}Let $N,c,e,\ep,a\in\bN$ such that $N\geqslant 2$, $c\geqslant \frac{3N-2}{4}$, $\ep\geqslant 1$ and $e\geqslant N+1+a+N\ep$, set $q:=e- (N+1+a+N\ep)$. Set $n:=N-c$. For any $0\leqslant j\leqslant N$ take $s^j\in \bA_\ep^{\oplus N+1}$ such that $X:=X_{s^1}\cap \cdots \cap X_{s^c}$ is a smooth complete intersection variety. 
 For any $i\in \{0,\dots,N\}$, set $H_i:=(Z_i=0)$ and $W_i:=X\cap H_i$. We look at $\bP(\Om_{W_i})$ as a subvariety of $\bP(\Om_X)$. Then, for a general choice of  $s^1,\dots,s^c$, there exists $E\subseteq \bP(\Om_X)$ such that $\dim E \ls 0$ and  $$\bB(\bL_X(-a))\subseteq \bigcap_{\substack{I\in \{1,\dots, c\}^n_{\neq}\\ P\in\bC[Z_0,\dots, Z_N]_q}}\!\!\!\!\!\!\!Ò (\om^{I,P}=0)=\bigcup_{i=0}^N \bP(\Om_{W_i})\cup E.$$
 Where $\om^{I,P}\in H^0(\bP(\Om_X),\bL_X^n(-a))\cong H^0(X,S^n\Om_X(-a))$ is the symmetric differential form constructed in Proposition \ref{DiffForms} viewed as a global section of $\bL_X^n(-a).$
\end{lemma}

\begin{proof}
For simplicity, we only treat the chart $U_0$, the other charts are dealt with in the exact same way. Let us precise the notation of Proposition \ref{DiffForms}: for any $(u,z,\xi)\in \bA_\ep\times \bC^N\times \bP^{N-1}$ we set 
 \begin{eqnarray*}
 b_q(u,z):=z_qu(z),\ \text{and} \ \  \be_q(u,z,\xi):=z_qdu_z(\xi)+eu(z)\xi_q.
 \end{eqnarray*}   Moreover, for any $t=(t^1,\dots,t^c)=(t^j_i)^{1\ls j\ls c}_{1\ls i\ls N}\in (\bA_\ep^N)^c$ and for any $(z,\xi)\in \bC^N\times \bP^{N-1}$ we set 
$$B(t,z):=\left(\begin{array}{ccc}b_1(t_1^1,z)&\cdots & b_N(t_N^1,z)\\ \vdots & & \vdots \\ b_1(t_1^c,z)& \cdots &b_N(t_N^c,z)\end{array}\right)\ \ \text{and}\ \ 
B'(t,z,\xi):=\left(\begin{array}{ccc}\be_1(t_1^1,z,\xi)&\cdots & \be_N(t_N^1,z,\xi)\\ \vdots & & \vdots \\ \be_1(t_1^c,z,\xi)& \cdots &\be_N(t_N^c,z,\xi)\end{array}\right).$$
From now on we make the identification $\bP({\Om_{\bP^N}}|_{U_0})\cong \bC^{N}\times \bP^{N-1}.$ So that if for any $j\in \{1,\dots, c\}$ we take $s^j\in \bA_\ep^{N+1}$ such that $X_s=(F_{s^1}=0)\cap \cdots \cap (F_{s^c}=0)$ is a smooth complete intersection in $\bP^N$ and if for any $j\in \{1,\dots,c\}$, we let $t^j$ be the dehomogeneization on $s^j$, $f_{t^j}$ be the dehomogeneization of $F_{s^j}$ and $X_t=X_s$, then, one naturally has 
$$\bP(\Om_{X_t\cap U_0})=\left\{(x,\xi)\in \bC^N\times \bP^{N-1}\ \text{such that }\ \ \left\{\begin{array}{c}f_{t^{1}}(z)=\cdots=f_{t^{c}}(z)=0\\ df_{t^{1},z}(\xi)=\cdots=df_{t^{c},z}(\xi)=0\end{array}\right.\right\}\subseteq \bC^{N}\times \bP^{N-1}\cong \bP(\Om_{U_0})$$ 
Moreover, set $W:=\bigcup_{i=1}^N\left((z_i=0)\cap(\xi_i=0)\right)\cong \bigcup_{i=1}^N \bP(\Om_{H_i\cap U_0}),$ where $H_i=(Z_i=0).$ 
Our aim is to prove that, for a general $t\in \bA_{\ep}^{(N+1)c},$ there exists $E\subseteq \bP(\Om_{X_t\cap U_0})$ such that $\dim E\ls 0$ and $$ \bP(\Om_{X_t\cap U_0})\cap\bigcap_{I,Q}\left(\om^{I,Q}_0(t,z,\xi)=0\right)=\left(\bP(\Om_{X_t\cap U_0})\cap W\right)\cup E.$$
The first thing to observe is that if $(x,\xi)\in W$ then for any $I\in \{1,\dots,c\}^n_{\neq}$ and any $Q\in\bC[z_1,\dots,z_N]_{\ls q}$ one has $\om^{I,Q}(t,z,\xi)=0$, and therefore $W\subseteq \bigcap_{I,Q}\left(\om^{I,Q}_0(t,z,\xi)=0\right).$ The more difficult part is to prove that, generically, this is actually an equality up to a finite number of points.

Observe that if $\rk(B(t,z))<c$, then $\om_0^{I,Q}(t,z,\xi)=0$ for all $I$ and all $Q$. On the other hand, if $ \rk(B(t,z))=c$ then 
$$\om_0^{I,Q}(t,z,\xi)=0 \ \forall I, \ \forall Q \ \ \iff \ \ \rk\left(\begin{array}{c}B(t,z)\\ B'(t,z,\xi)\end{array}\right)<N.$$
Therefore, for any $t\in \bA_\ep^{(N+1)c},$ the locus $\bigcap_{I,Q}\left(\om^{I,Q}_0(t,z,\xi)=0\right)$ is precisely 
$$\left\{(z,\xi)\in \bC^N\times \bP^{N-1} \ \ \text{such that}\ \ \rk B(t,z)<c \ \ \text{or}\ \ \rk B(t,z)=c \ \text{and}\ \rk\left(\begin{array}{c}B(t,z)\\ B'(t,z,\xi)\end{array}\right)<N \right\}$$
 To understand this locus for general $t\in \bA_\ep^{(N+1)c}$ we will consider the problem in family, and study the corresponding incidence varieties.
 Let us introduce some more notation. Consider the natural projections 
\begin{eqnarray*}
\rho_1:\bA_\ep^{(N+1)c}\times \bC^N \to \bA_\ep^{(N+1)c}\ \ \ \ \ \rho_2:\bA_\ep^{(N+1)c}\times \bC^N \to \bC^N\ \ \ \ \ p_{1}:\bA_\ep^{(N+1)c}\times \bC^N \times \bP^{N-1}\to \bA_\ep^{(N+1)c}\\
p_{12}:\bA_\ep^{(N+1)c}\times \bC^N \times \bP^{N-1}\to \bA_\ep^{(N+1)c}\times \bC^N\ \ \ \ \ p_{23}:\bA_\ep^{(N+1)c}\times \bC^N\times \bP^{N-1} \to \bC^N\times \bP^{N-1}.
\end{eqnarray*}
For any $k\in \{0,\dots, N\},$ set $Y_k=\left\{(z_1,\dots,z_N)\in\bC^N \ \ /\ \ \exists I\in \{1,\dots,N\}^k_{\neq} \  \text{such that} \ \ z_i=0\iff i\in I \right\}.$
Observe that $\dim Y_k=N-k.$ Moreover, set 
\begin{eqnarray*}
\ccX_c&:=&\left\{\left((t^{1},\dots,t^{c}),z\right)\in (\bA_\ep^{(N+1)})^c\times \bC^N\ \text{such that}\ \ f_{t^{1}}(z)=\cdots=f_{t^{c}}(z)=0\right\}\\
\ccX_c'&:=&\left\{\left((t^{1},\dots,t^{c}),z,\xi\right)\in (\bA_\ep^{(N+1)})^c\times \bC^N\times \bP^{N-1}\ \text{such that}\ \ \left\{\begin{array}{c}f_{t^{1}}(z)=\cdots=f_{t^{c}}(z)=0\\ df_{t^{1},z}(\xi)=\cdots=df_{t^{c},z}(\xi)=0\end{array}\right.\right\}\\
\De_c&:=&\left\{\left(t,z\right)\in \ccX_c \ \text{such that}\ \ \rk B(t,z)<c\right\}\\
E_c^0&:=&\left\{\left(t,z,\xi\right)\in \ccX_c' \ \text{such that}\ \ \rk B(t,z)=c \ \text{and}\ \ \rk\left(\begin{array}{c}B(t,z)\\ B'(t,z,\xi)\end{array}\right)<N\right\}\\
\De_c'&:=&p_{12}^{-1}(\De_c)\cap\ccX_c',\ \ \ E_c:=E^0_c\cup \De_c', \ \ \ \Ga_c:=E_c\setminus p_{23}^{-1}(W) \ \text{and} \ \ \ \Ga_c^{0}:=E_c^0\setminus p_{23}^{-1}(W).
\end{eqnarray*}
Of course, $\ccX_c\stackrel{\rho_1}{\to}\bA_\ep^{(N+1)c}$ is just the universal family of complete intersection of Fermat type we are considering (or to be precise, the dehomogeneization of it), and $\ccX_c'$ is just the projectivization of the relative cotangent bundle of $\ccX_c\stackrel{\rho_1}{\to}\bA_\ep^{(N+1)c}.$ For any $t\in \bA_\ep^{(N+1)c},$ we write $\ccX_{c,t}=\rho_1^{-1}(t)$ (and similarly for $\ccX'_{c,t},$ $\De_{c,t}$, etc.). By the above discussion, our aim is to prove that for a general $t\in \bA_\ep^{(N+1)c},$ $\Ga_{c,t}$ is finite. To do so, we will prove that for a general $t\in \bA_{\ep}^{(N+1)c}$, $\De'_{c,t}$ is empty and $\dim \Ga_{c,t}^0\ls 0.$ This will follow at once from the two following claims.
\begin{claim}\label{claim2}$\dim \De_c<\dim\bA_\ep^{(N+1)c}$\end{claim}
\begin{claim}\label{claim3}$\dim \Ga_c^{0}\ls\dim\bA_\ep^{(N+1)c}$\end{claim}
\begin{proof}[Proof of Claim \ref{claim2}] Take $k\in \{0,\dots, N\}$. Fix $(z_1,\dots, z_N)\in Y_k.$ Without loss of generality, we may assume that $z_1=\cdots =z_k=0$ and that $z_{k+1}\cdots z_N\neq 0.$ In particular, 
$$\rk B(t,z)=\rk\left(\begin{array}{ccc}t_{k+1}^1(z)&\cdots & t_N^1(z)\\ \vdots & &\vdots \\ t_{k+1}^c(z)&\cdots & t_N^c(z)\end{array}\right).$$
Set $\De_{c}^z:=\rho_1\left(\rho_2^{-1}(\{z\})\cap \De_c\right).$ Of course, 
$$\De_{c}^z=\left\{\left((t_i^j)^{1\ls j\ls c}_{1\ls i\ls N}\right)\in \bA_\ep^{(N+1)c}\ \text{such that}\ \ \rk\left(\begin{array}{ccc}t_{k+1}^1(z)&\cdots & t_N^1(z)\\ \vdots & &\vdots \\ t_{k+1}^c(z)&\cdots & t_N^c(z)\end{array}\right)<c,\ \text{and}\ f_{t^1}(z)=\cdots=f_{t^c}(z)=0\right\}.$$
By Lemma \ref{determinental} the set defined by $\rk\left(\begin{array}{ccc}t_{k+1}^1(z)&\cdots & t_N^1(z)\\ \vdots & &\vdots \\ t_{k+1}^c(z)&\cdots & t_N^c(z)\end{array}\right)<c$ is of dimension less than  $\dim \bA_\ep^{(N+1)c}-\max\{0,N-c-k+1\}$. Therefore, $$\dim \De_{c}^z\ls\dim\bA_\ep^{(N+1)c}-\max\{0,N-k-c+1\}-c.$$
Indeed, for any $j\in \{1,\dots, c\},$ the equation $f_{t^j}(z)=t_0^j(z)+t_1^j(z)z_1^e+\cdots +t_N^j(z)z_N^e=0$ is affine (in $t_i^j$) and involves the term $t_0^j(z)$ which appears in none of the other equations defining $\De_{c}^z$. So that each of the equations $f_{t^j}(z)=0$ increases the codimension by one, and we get the announced dimension. Therefore by Lemma \ref{stratification} we obtain that 
\begin{eqnarray*}
\dim \De_c &\ls& \max_{0\leqslant k\ls N}\left(\dim Y_k+\dim\bA_\ep^{(N+1)c}-\max\{0,N-k-c+1\}-c\right)\\
&=& \max_{0\leqslant k\ls N}\left(\dim\bA_\ep^{(N+1)c}-\max\{0,N-k-c+1\}+N-k-c\right)\ls \dim\bA_\ep^{(N+1)c}-1<\dim \bA_\ep^{(N+1)c}.
\end{eqnarray*}
\end{proof}
To prove Claim \ref{claim3} we will need the following observation. 
\begin{claim}\label{claim0}Let $(x,\xi)\in\bC^N\times\bP^{N-1}\setminus W$. For any $q\in \{1,\dots, N\}$, $\be_q(\cdot,z,\xi)\neq 0$ and $\ker b_q(\cdot, z)\neq \ker\be_q(\cdot,z,\xi).$\end{claim}
\begin{proof}[Proof of Claim \ref{claim0}]
For $k\in\{1,\dots,N\}$, we set $\de_k=(0,\dots,0,1,0,\dots, 0)\in \bN^N$ be the multi-index whose only non-zero term is in the $k$-th slot. For any $u\in \bA_\ep$ any $(z,\xi)\in \bC^N\times \bP^{N-1}$ and any $q\in \{1,\dots,N\}$ we have by definition,
\begin{eqnarray*}
b_q(u,z)&=&z_q\left(\sum_{|I|\ls \ep}u_Iz^I\right)=z_qu_0+z_1z_qu_{\de_1}+\cdots +z_Nz_qu_{\de_N}+\sum_{2\ls |I|\ls \ep}z^{I+\de_q}u_I\\ 
\be_q(u,z,\xi)&=& z_q\left(\sum_{|I|\ls \ep}\sum_{k=1}^Ni_ku_Iz^{I-\de_k}\xi_k\right)+e\xi_q\left(\sum_{|I|\ls \ep}u_Iz^I\right)\\
&=&e\xi_qu_0+(z_q\xi_1+ez_1\xi_q)u_{\de_1}+\cdots+(z_q\xi_N+ez_N\xi_q)u_{\de_N}+\sum_{2\ls |I|\ls \ep}\left(z_q\sum_{k=1}^Ni_kz^{I-\de_k}\xi_k+ez^I\xi_q\right)u_I.
\end{eqnarray*} 
Take $(x,\xi)\in \bC^{N-1}\times \bP^{N-1}\setminus W.$ Let $q\in \{1,\dots, N\}$. In view of the above expression, if $\be_q=0$ then in particular $$e\xi_q=z_q\xi_1+ez_1\xi_q=\cdots=z_q\xi_N+ez_N\xi_q=0,$$
and because $(z,\xi)\notin W$, we get $z_q\neq 0$, so that $\xi_q=\xi_1=\cdots=\xi_N=0$, which is not possible because $\xi\in\bP^{N-1}.$ It remains to prove that $\ker b_q(\cdot,z)\neq \ker \be_q(\cdot,z,\xi).$ If $b_q(\cdot,z)=0$, this is obvious. Suppose $b_q(\cdot,z)\neq 0$ and therefore $z_q\neq 0$. If $\ker b_q(\cdot,z)=\ker \be_q(\cdot,z,\xi),$ then  from the above expression we get in particular that
$$\left|\begin{array}{cc}1& z_1\\ e\xi_q& z_q\xi_1+ez_1\xi_q\end{array}\right|=\cdots=\left|\begin{array}{cc}1& z_N\\ e\xi_q& z_q\xi_N+ez_N\xi_q\end{array}\right|=0,$$
From which we deduce once again that $\xi_1=\cdots=\xi_N=0.$
\end{proof}

\begin{proof}[Proof of Claim \ref{claim3}] For any $K=(k_1,\dots, k_c)\in \{1,\dots, N\}^c_{\neq}$ let 
$$U^K=\left\{(t,z,\xi)\in \bA_\ep^{(N+1)c}\times \bC^N\times \bP^{N-1}\ \text{such that }\ \det\left(\begin{array}{ccc}b_{k_1}(t_{k_1}^1,z)&\cdots &b_{k_c}(t_{k_c}^1,z)\\ \vdots&&\vdots\\ b_{k_1}(t_{k_1}^c,z)&\cdots &b_{k_c}(t_{k_c}^c,z)\end{array}\right)\neq 0\right\}.$$
For any $K$, $U^K_\Ga:=U^K\cap\Ga_c^0$ is an open subset of $\Ga_c^{0}$ and $\Ga_c^{0}=\bigcup_KU^K_\Ga,$ therefore for our dimension estimation, we might as well restrict ourselves to $U:=U^{(1,\dots,c)}$ and $U_\Ga:=U_\Ga^{(1,\dots,c)}.$ Fix $(z,\xi)\in \bC^N\times\bP^{N-1}$ and set $U^{z,\xi}_\Ga:=p_1(p_{23}^{-1}(z,\xi)\cap U_\Ga)$. Using Claim \ref{claim0} and Lemma \ref{FormesIndep} we see that the set defined in $\bA_\ep^{(N+1)c}$ by 
$$ \rk\left(\begin{array}{c}B(t,z)\\ B'(t,z,\xi)\end{array}\right)<N\ \ \text{and}\ \ \ (t,z,\xi)\in U$$
is of dimension less than $\dim\bA_\ep^{(N+1)c}-2c+N-1$, and by the same argument as in the proof of Claim \ref{claim2}, we obtain
$$\dim U_\Ga^{z,\xi}\ls\dim\bA_{\ep}^{(N+1)c}-2c+N-1-2c.$$
Hence $\dim U_\Ga\ls\dim\bA_\ep^{(N+1)c}-4c+N-1+\dim(\bC^N\times\bP^{N-1})=\dim\bA_\ep^{(N+1)c}-4c+N-1+2N-1.$ And therefore (because the same argument holds for all the $U^K_\Ga$)
$$\dim \Ga_c^{0}=\dim\bA_\ep^{(N+1)c}-4c+3N-2.$$
In particular, $\dim \Ga_c^{0}\ls\dim\bA_\ep^{(N+1)c}$ as soon as $-4c+3N-2\ls 0$. This last condition is equivalent to $c\gs \frac{3N-2}{4},$ which is exactly our hypothesis.
\end{proof}
\end{proof}


\bibliographystyle{plain}      
\bibliography{biblio.bib}   

\begin{thebibliography}{10}

\bibitem{Bog78}
F.~A. Bogomolov.
\newblock Holomorphic symmetric tensors on projective surfaces.
\newblock {\em Uspekhi Mat. Nauk}, 33(5(203)):171--172, 1978.

\bibitem{BDO08}
Fedor Bogomolov and Bruno De~Oliveira.
\newblock Symmetric tensors and geometry of {$\Bbb P^N$} subvarieties.
\newblock {\em Geom. Funct. Anal.}, 18(3):637--656, 2008.

\bibitem{Bro14}
Damian Brotbek.
\newblock Hyperbolicity related problems for complete intersection varieties.
\newblock {\em Compos. Math.}, 150(3):369--395, 2014.

\bibitem{B-R90}
P.~Br{\"u}ckmann and H.-G. Rackwitz.
\newblock {$T$}-symmetrical tensor forms on complete intersections.
\newblock {\em Math. Ann.}, 288(4):627--635, 1990.

\bibitem{Bru85}
Peter Br{\"u}ckmann.
\newblock Some birational invariants of algebraic varieties.
\newblock In {\em Proceedings of the conference on algebraic geometry
  ({B}erlin, 1985)}, volume~92 of {\em Teubner-Texte Math.}, pages 65--73.
  Teubner, Leipzig, 1986.

\bibitem{BKT13}
Yohan Brunebarbe, Bruno Klingler, and Burt Totaro.
\newblock Symmetric differentials and the fundamental group.
\newblock {\em Duke Math. J.}, 162(14):2797--2813, 2013.

\bibitem{Deb05}
Olivier Debarre.
\newblock Varieties with ample cotangent bundle.
\newblock {\em Compos. Math.}, 141(6):1445--1459, 2005.

\bibitem{Div08}
Simone Diverio.
\newblock Differential equations on complex projective hypersurfaces of low
  dimension.
\newblock {\em Compos. Math.}, 144(4):920--932, 2008.

\bibitem{Div09}
Simone Diverio.
\newblock Existence of global invariant jet differentials on projective
  hypersurfaces of high degree.
\newblock {\em Math. Ann.}, 344(2):293--315, 2009.

\bibitem{DMR10}
Simone Diverio, Jo{\"e}l Merker, and Erwan Rousseau.
\newblock Effective algebraic degeneracy.
\newblock {\em Invent. Math.}, 180(1):161--223, 2010.

\bibitem{Har77}
Robin Hartshorne.
\newblock {\em Algebraic geometry}.
\newblock Springer-Verlag, New York, 1977.
\newblock Graduate Texts in Mathematics, No. 52.

\bibitem{Laz04I}
Robert Lazarsfeld.
\newblock {\em Positivity in algebraic geometry. {I}}, volume~48 of {\em
  Ergebnisse der Mathematik und ihrer Grenzgebiete. 3. Folge. A Series of
  Modern Surveys in Mathematics [Results in Mathematics and Related Areas. 3rd
  Series. A Series of Modern Surveys in Mathematics]}.
\newblock Springer-Verlag, Berlin, 2004.
\newblock Classical setting: line bundles and linear series.

\bibitem{Laz04II}
Robert Lazarsfeld.
\newblock {\em Positivity in algebraic geometry. {II}}, volume~49 of {\em
  Ergebnisse der Mathematik und ihrer Grenzgebiete. 3. Folge. A Series of
  Modern Surveys in Mathematics [Results in Mathematics and Related Areas. 3rd
  Series. A Series of Modern Surveys in Mathematics]}.
\newblock Springer-Verlag, Berlin, 2004.
\newblock Positivity for vector bundles, and multiplier ideals.

\bibitem{Mer13}
J.~{Merker}.
\newblock {Siu-Yeung jet differentials on complete intersection surfaces X\^{}2
  in P\^{}4(C)}.
\newblock {\em ArXiv e-prints}, December 2013.

\bibitem{Mer14}
J.~{Merker}.
\newblock {Extrinsic projective curves X\^{}1 in P\^{}2(C): harmony with
  intrinsic cohomology}.
\newblock {\em ArXiv e-prints}, February 2014.

\bibitem{R-R14}
Xavier Roulleau and Erwan Rousseau.
\newblock Canonical surfaces with big cotangent bundle.
\newblock {\em Duke Math. J.}, 163(7):1337--1351, 2014.

\bibitem{Sch85}
Michael Schneider.
\newblock Complex surfaces with negative tangent bundle.
\newblock In {\em Complex analysis and algebraic geometry ({G}\"ottingen,
  1985)}, volume 1194 of {\em Lecture Notes in Math.}, pages 150--157.
  Springer, Berlin, 1986.

\bibitem{Siu04}
Yum-Tong Siu.
\newblock Hyperbolicity in complex geometry.
\newblock In {\em The legacy of {N}iels {H}enrik {A}bel}, pages 543--566.
  Springer, Berlin, 2004.

\bibitem{S-Y96}
Yum-Tong Siu and Sai-kee Yeung.
\newblock Hyperbolicity of the complement of a generic smooth curve of high
  degree in the complex projective plane.
\newblock {\em Invent. Math.}, 124(1-3):573--618, 1996.

\bibitem{Som84}
Andrew~John Sommese.
\newblock On the density of ratios of {C}hern numbers of algebraic surfaces.
\newblock {\em Math. Ann.}, 268(2):207--221, 1984.

\end{thebibliography}

\end{document}